\documentclass{article}

\usepackage{hyperref}

\usepackage{setup}
\usepackage{soul}   
\usepackage{hyperref}  
\usepackage{booktabs}  
\usepackage{amsthm}  
\usepackage{placeins}  
\usepackage{float}
\usepackage{txfonts}
\usepackage{mathrsfs}
\restylefloat{table}
\usepackage{enumitem}

\usepackage{amsmath}  
\usepackage{amssymb}

\usepackage{xcolor}  
\usepackage{tcolorbox}

\usepackage{tikz}  
\usetikzlibrary{calc}
\usepackage{latexsym}
\usepackage{a4wide}
\usepackage{pgfplots}
\usepackage{tikz-cd}
\usepgfplotslibrary{polar}
\pgfplotsset{compat=newest}

\newtheorem{theorem}{Theorem}[section]  
\newtheorem{proposition}[theorem]{Proposition}
\newtheorem{lemma}[theorem]{Lemma}
\newtheorem{corollary}[theorem]{Corollary}
\theoremstyle{definition}
\newtheorem{definition}[theorem]{Definition}
\newtheorem{remark}[theorem]{Remark}

\newtheorem{example}[theorem]{Example}

\usepackage[all]{xy}  

\usepackage{pifont}  

\begin{document}

\title{Generic  infinitesimal rigidity for rotational groups in the plane}
\author{Alison La Porta\thanks{School of Mathematical Sciences, Lancaster University, UK, \texttt{a.laporta@lancaster.ac.uk} (corr. author)} \,
and Bernd Schulze\thanks{School of Mathematical Sciences, Lancaster University, UK, \texttt{b.schulze@lancaster.ac.uk}}}

\date{}
\maketitle

\begin{abstract}
In this paper we establish combinatorial characterisations of symmetry-generic infinitesimally rigid frameworks in the Euclidean plane for rotational groups of order $4$ and $6$, and of odd order between $5$ and $1000$, where a joint may lie at the centre of rotation. This extends the corresponding results for these groups in the free action case obtained by R.~Ikeshita and S.~Tanigawa in 2015, and our recent results for the reflection group and the rotational groups of order $2$ and $3$ in the non-free action case. The characterisations are given in terms of sparsity counts on the corresponding group-labeled quotient graphs, and are obtained via symmetry-adapted versions of recursive Henneberg-type graph constructions.  
For rotational groups of even order at least $8$, we show that the sparsity counts alone are not sufficient for symmetry-generic infinitesimal rigidity. 

\end{abstract}

\noindent \textbf{Keywords}: infinitesimal rigidity; rotational symmetry; non-free group action; gain graph; sparsity counts.

\section{Introduction}

The rigidity and flexibility analysis of symmetric bar-joint frameworks and related geometric constraint systems has received a lot of attention  over the last two decades or so, which has led to an explosion of results on this topic; see for example \cite{sw2011,jzt2016,bernstein,mtforced,tanigawamatroids,dewar,kkm,bt2015,Ikeshita,RIandST16,schtan_bodies}. This surge in activity was in  part motivated by modern practical applications of the theory in areas such as structural engineering, robotics, structural biology, materials science, and computer-aided design, where non-trivial symmetries are ubiquitous and often have a crucial impact on the properties and behaviour of the structures.

A major line of research in this area has been to obtain efficient methodologies for determining when a bar-joint framework that is generic with respect to certain prescribed symmetry constraints is infinitesimally (or equivalently statically) rigid. Using methods from group representation theory, necessary conditions for a symmetric framework in Euclidean d-space to be ``iostatic" (i.e. minimally infinitesimally rigid, in the sense that the removal of any edge yields an infinitesimally flexible framework) were established in \cite{FowlerGuest02,sch10a,CFGSW09}. In  \cite{SLaman2,SLaman},  these conditions were shown to be sufficient if the framework is generic with respect to a group generated by a reflection, half-turn or three-fold rotation in the plane. See \cite{bernd2017sym,KTHandbookSch} for further results and open questions regarding symmetric isostatic frameworks.

The more general question of when a symmetry-generic framework is infinitesimally rigid, rather than just isostatic, is more complicated, as not every symmetric infinitesimally rigid framework has an isostatic spanning subframework with the same symmetry. A major breakthrough for analysing this more general question was achieved in \cite{bt2015}. The main idea in that work was to decompose the rigidity matrix (a fundamental tool for analysing infinitesimal rigidity) into block matrices corresponding to the irreducible representations of the group, and to set up a simpler but equivalent ``orbit rigidity matrix”  for each block. The underlying combinatorial structure for each orbit rigidity matrix is a group-labelled quotient graph, also known as a ``gain graph”, and using symmetry-adapted Henneberg-type recursive graph constructions, one can try to characterise the gain graphs that yield orbit rigidity matrices of maximum rank (corresponding to frameworks with only trivial phase-symmetric infinitesimal motions) for symmetry-generic configurations. Using this approach, combinatorial characterisations for symmetry-generic infinitesimal rigidity in the plane have been obtained for the groups generated by a reflection, half-turn and three-fold rotation in  \cite{bt2015}. R. Ikeshita and S. Tanigawa then extended these results further to rotational groups of odd order at most 1000 in the plane  \cite{Ikeshita}. See also \cite{RIandST16}.

 Since there is no combinatorial characterisation for generic rigidity without symmetry in dimensions 3 and higher, there are no analogous results for symmetric bar-joint frameworks for $d\geq 3$. However, such results have been established for the special classes of body-bar and body-hinge frameworks with $\mathbb{Z}_2\times \cdots \times \mathbb{Z}_2$ symmetry in \cite{schtan_bodies}. See also the result on symmetric $d$-pseuodmanifolds in $(d+1)$-space obtained in \cite{crjata}.
 
Importantly, all of the above results on symmetry-generic infinitesimal rigidity have made the assumption that the symmetry group acts freely on the vertex set of the graph. This simplifies the structure of the orbit rigidity matrices and the corresponding sparsity counts for the gain graphs appearing in the combinatorial characterisations significantly. However, this means that our knowledge of when symmetric frameworks are infinitesimally rigid is still severely restricted. Frameworks that model real-world structures in some of the application areas mentioned above are often symmetric, with vertices lying on reflection lines or centers of rotation \cite{smmb,schmil,stz}. Thus, closing this gap in our understanding of symmetric frameworks is not just of mathematical interest, but also important for a variety of real-world applications.

In our recent paper \cite{SchulzeLP2024}, we introduced orbit rigidity matrices for symmetric frameworks in the plane, where the group action is not necessarily free on the vertex set. Moreover, we used these matrices to establish necessary conditions for infinitesimal rigidity and to extend the combinatorial characterisations obtained in \cite{bt2015} for the reflection, half-turn and three-fold rotational group in the plane to the cases when vertices may be fixed by non-trivial group elements. Here we extend these results further to all the groups considered in \cite{Ikeshita}, i.e. all rotational groups of odd order at most 1000, where a vertex may be fixed by a rotation. As we will see, the gain sparsity counts for these groups are even more complex than for the groups of order at most $3$. In addition, we provide analogous results for the rotational groups of order $4$ and $6$, and we provide classes of examples, both for the free and non-free group action case, which show that the standard sparsity counts are not sufficient for symmetry-generic infinitesimal rigidity.

The paper is organised as follows. In Section~\ref{sec:infrig} we review basic notions from rigidity theory. Section~\ref{sec:gaingraphs} introduces gain graphs and provides  the relevant gain sparsity counts that are necessary for infinitesimal rigidity. Section~\ref{sec:red} introduces operations on gain graphs that preserve phase-symmetric infinitesimal rigidity. These operations are used in Sections~\ref{sec:blockers}, \ref{sec:reduct} and \ref{sec:final} to show the sufficiency of the gain sparsity counts for the groups of order $4$ and $6$ and of odd order at most $1000$ via an induction proof. Finally, in Section~\ref{sec:even k} we provide examples of gain graphs for the cyclic groups of even order at least $8$, which satisfy all the necessary gain sparsity counts, but still correspond to infinitesimally flexible frameworks.

\section{Infinitesimal rigidity of symmetric frameworks}\label{sec:infrig}

A \textit{(bar-joint) framework} in $\mathbb{R}^d$ is a pair $\left(G,p\right)$ where $G$ is a finite simple graph and $p:V\left(G\right)\rightarrow\mathbb{R}^d$ is an injective map. We say $(G,p)$ and $p$ are a \textit{realisation} and a \textit{configuration}, respectively, of the \textit{underlying graph} $G$. We will assume throughout the paper that $p(V(G))$ affinely spans $\mathbb{R}^d$. An \textit{infinitesimal motion} of $(G,p)$ is a function $m:V(G)\rightarrow\mathbb{R}^d$ such that for all $\{u,v\}\in E(G),$ 
\begin{equation}
\label{inf. motion}
     (p(u)-p(v))^T\cdot(m(u)-m(v))=0.
\end{equation}
The infinitesimal motion $m$ is defined to be \textit{trivial} if there is a skew-symmetric matrix $M\in M_d(\mathbb{R})$ and a $d$-dimensional vector $t$ such that $m(u)=Mp(u)+t$ for all $u\in V(G)$. We say $(G,p)$ is \textit{infinitesimally rigid} if all of its infinitesimal motions are trivial. It is often useful to view $m$ as a column vector of size $d|V(G)|.$ When doing so, the space of infinitesimal motions of a frameworks coincide with the right kernel of a well-known matrix, the \textit{rigidity matrix} of $(G,p)$, which we usually denote $R(G,p)$. It is easy to see that the space of trivial infinitesimal motions of a framework that affinely spans $\mathbb{R}^d$ has dimension $\frac{d(d+1)}{2}$. Hence, $\textrm{null}\,R(G,p)\geq\frac{d(d+1)}{2}$, and a framework is infinitesimally rigid if and only if this equation holds with equality.


This paper is concerned with frameworks which are symmetric with respect to a rotational group. Here, we formalise the definition of symmetric graphs, and hence the definition of symmetric frameworks. First, we set some group notation that will be used throughout the paper.

Let $k\geq4$ be an integer, and $\mathbb{Z}_k=\{0,1,\dots,k-1\}$ be the additive cyclic group of order $k$. We often identify $\mathbb{Z}_k$ with the multiplicative group $\Gamma=\left<\gamma\right>$ through the isomorphism which maps 1 to $\gamma$. For $0\leq j\leq k-1$, we use $\rho_j$ to denote the group representation of $\Gamma$ which sends $\gamma$ to the scalar $\exp{\frac{2\pi ij}{k}}$. We use $C_k$ to denote the anti-clockwise rotation around the origin by $\frac{2\pi}{k}$, and $\mathcal{C}_k$ to denote the group generated $C_k$. We also use $\tau_k:\Gamma\rightarrow\mathcal{C}_k$ to denote the isomorphism which maps $\gamma$ to $C_k$.

We say a finite simple graph $G$ is $\Gamma$-symmetric if there is a homomorphism $\theta:\Gamma\rightarrow\textrm{Aut}(G)$, where $\textrm{Aut}(G)$ denotes the automorphism group of $G$. Notice that, since $\Gamma\simeq\mathbb{Z}_k$, we may also say $G$ is $\mathbb{Z}_k$-symmetric. We will often drop the map $\theta$ from the notation if it's clear from the context, and abbreviate $\theta(\delta)$ to $\delta$, for all $\delta\in\Gamma$. We say a framework $(G,p)$ in $\mathbb{R}^2$ is $\mathcal{C}_k$-symmetric if $G$ is $\Gamma$-symmetric and, for all $\delta\in\Gamma,v\in V(G)$, we have $\tau_k(\delta)p(v)=p(\delta v)$. 
We say $p$ (or, equivalently $(G,p)$) is \textit{$\mathcal{C}_k$-generic} if $\textrm{rank}\,R(G,p)\geq\textrm{rank}\,R(G,q)$ for all  realisations $(G,q)$ that are $\mathcal{C}_k$-symmetric (with the same $\theta$). 

Given a $\Gamma$-symmetric graph $G$ and a vertex $v\in V(G)$, we say $\delta\in\Gamma$ \textit{fixes} $v$ if $\delta v=v$, and we use $S_{\Gamma}(v)$ to denote the subgroup of $\Gamma$ whose elements are exactly the elements which fix $v$. We define the elements of $V_0(G):=\{v\in V(G):S_{\Gamma}(v)=\Gamma\}$ and $\overline{V(G)}:=\{v\in V(G):S_{\Gamma}(v)=\{\textrm{id}\}\}$ to be the \textit{fixed} and \textit{free vertices} of $G$, respectively. Let $(G,p)$ be a $\mathcal{C}_k$-symmetric framework, and $v\in V(G)$ be fixed by a non-identity element $\delta\in\Gamma$. Since $\tau_k(\delta)$ is a non-trivial rotation, it is easy to see that $p(v)$ is the zero vector, and $\tau_k(\delta')p(v)=p(v)$ for all $\delta'\in\Gamma$. Since we are concerned with the $\mathcal{C}_k$-symmetric realisations of $\Gamma$-symmetric graphs, we assume throughout the paper that $V(G)=V_0(G)\,\dot\cup\,\overline{V(G)}$ and that $|V_0(G)|\leq1$ (recall that $p$ is an injective function). 

Theorem 3.2 in \cite{sch10a} shows that, under a suitable symmetry adapted basis, the rigidity matrix of a $\mathcal{C}_k$-symmetric framework $(G,p)$ block diagonalises into $k$ matrices $\tilde{R}_0(G,p),\dots,\tilde{R}_{k-1}(G,p)$, each one corresponding to an irreducible representation $\rho_j$ of the cyclic group $\Gamma$ of order $k$. Given $0\leq j\leq k-1$, an infinitesimal motion $m$ of $(G,p)$ is said to be \textit{$\rho_j$-symmetric} if $m(\delta v)=\overline{\rho_j(\delta)}\tau_k(\delta)m(v)$ for all $v\in V(G)$ and all $\delta\in\Gamma$. We say $(G,p)$ is \textit{$\rho_j$-symmetrically isostatic} if  all $\rho_j$-symmetric infinitesimal motions of $(G,p)$ are trivial and $R_j(G,p)$ has no non-trivial row dependence. Notice that, if $(G,p)$ is $\rho_j$-symmetrically isostatic for some $0\leq j\leq k-1$, then any $\mathcal{C}_k$-generic realisation $(G,q)$ of $G$ is $\rho_j$-symmetrically isostatic. 

Let $(G,p)$ be a $\mathcal{C}_k$-symmetric framework in $\mathbb{R}^2$. Recall that the nullity of $R(G,p)$ is at least 3. The following result shows how the null space of $R(G,p)$ splits with respect to the block-diagonalisation of the rigidity matrix (for an argument, see the proof Theorem 6.7 in \cite{bt2015}).

\begin{proposition}
Let $k\geq4$, and $(G,p)$ be a $\mathcal{C}_k$-symmetric framework. The spaces of trivial $\rho_0$-,$\rho_1$- and $\rho_{k-1}$-symmetric infinitesimal motions all have dimension 1. For $2\leq j\leq k-2$, the space of trivial $\rho_j$-symmetric infinitesimal motions has dimension 0.
\end{proposition}

Clearly, a $\mathcal{C}_k$-symmetric framework is infinitesimally rigid if and only if it has no non-trivial $\rho_j$-symmetric infinitesimal motion for all $0\leq j\leq k-1$. Hence, we aim to combinatorially characterise $\mathcal{C}_k$-generic $\rho_j$-symmetrically isostatic frameworks for each $\rho_j$ in order to have a characterisation of $\mathcal{C}_k$-generic infinitesimally rigid  frameworks. In \cite{SchulzeLP2024}, we characterised $\mathcal{C}_k$-generic $\rho_0$-,$\rho_1$- and $\rho_{k-1}$-symmetrically isostatic frameworks for cyclic groups, and we found necessary conditions for $\rho_j$-symmetric isostaticity for all $2\leq j\leq k-2$. This paper is aimed at giving sufficient conditions for $\rho_j$-symmetric isostaticity, where $2\leq j\leq k-2$, in order to have a full characterisation of  $\mathcal{C}_k$-generic infinitesimally rigid frameworks.

The main combinatorial object we will use for this is the gain graph, which allows us to reduce the redundancies which occur in symmetric frameworks. As we will see, gain graphs are labelled multigraphs. Hence, a vertex may have a loop. We assume, throughout the paper, that a loop adds 2 to the degree of a vertex.

\section{Gain graphs} \label{sec:gaingraphs}

For an integer $k\geq4$, let $(\tilde{G},\tilde p)$ be a $\mathcal{C}_k$-symmetric framework and consider its underlying $\Gamma$-symmetric graph $\tilde G$. Let $G$ be the $\Gamma$-quotient graph of $\tilde G$, whose vertex set and edge set are, respectively, the sets of vertex orbits and edge orbits of $\tilde G$. Recall that, if $\tilde{G}$ has a fixed vertex, then $V(\tilde{G})$ is partitioned into the set of free vertices of $\tilde G$ and the singleton set containing the fixed vertex of $\tilde G$. It follows that $V(G)$ is partitioned into the sets $V_0(G):=\{v\in V(G):|v|=1\}$ and $\overline{V(G)}:=\{v\in V(G):|v|=k\}$, where $|V_0(G)|\leq1$. 

By orienting the edges of $G$ and assigning them a group label, we create a combinatorial object, known as the ``$\Gamma$-gain graph" of $\tilde G$, which contains all of the information of $\tilde G$, and discards any redundancies. We do so in the following way. 

First, we fix an orientation on the edges of $G$. Then, for each vertex orbit $v\in V(G)$, we fix a representative vertex $v^{\star}\in V(\tilde G)$. We define the following \textit{gain function} $\psi:E(G)\rightarrow\Gamma$. For each directed edge $e= (u,v)$:
\begin{itemize}
    \item If $u,v\in \overline{V(G)}$, then there exists a unique $\delta\in\Gamma$ such that $\{u^{\star},\delta v^{\star}\}\in e$. We let $\psi(e)=\delta$.
    \item If one of $u,v$ is fixed, say $u\in V_0(G)$, then $e=\{u^{\star},\delta v^{\star}|\, \delta\in \Gamma\}$. We let $\psi(e)=\delta$ for any $\delta\in\Gamma$. 
\end{itemize}
We say $(G,\psi)$ is the \textit{$\Gamma$-gain graph} of $\tilde G$, and we say $\tilde G$ is the \textit{$\Gamma$-lifting} (or \textit{$\Gamma$-covering}) of $(G,\psi)$. 

Let $p:V(G)\rightarrow\mathbb{R}^2$ be defined by letting $p(u)=\tilde p(u^{\star})$ for all $u\in V(G)$. Then, we say $(G,\psi,p)$ is the \textit{$\mathcal{C}_k$-gain framework} of $(\tilde G,\tilde p).$ For each block $\tilde R_j(\tilde G,\tilde p)$ in the rigidity matrix, we may construct a matrix $O_j(G,\psi,p)$ of the same size and the same rank and nullity, which solely depends on the $\mathcal{C}_k$-gain framework $(G,\psi,p)$ (see Section 4 in \cite{SchulzeLP2024} for the definition of $O_j(G,\psi,p)$ and for more details).

In this construction, we can redirect any edge and label it with the group inverse of the original label chosen. Up to this operation, up to the choice of representatives, and up to the choice of labels on the edges incident to the fixed vertex, this process gives rise to a unique $\Gamma$-gain graph. Two $\Gamma$-gain graphs of the same $\Gamma$-symmetric graph are called \textit{equivalent}. Equivalent $\Gamma$-gain graphs share the same combinatorial properties. Moreover, since they share the same $\Gamma$-lifting, they also have the same infinitesimal rigidity properties (see Lemma 4.6 and Proposition 5.1 in \cite{SchulzeLP2024}).  The following was shown in \cite{jzt2016} (Lemma 2.4) for the case where $V_0(G)=\emptyset$, and the same argument can easily be generalised for the case where $V_0(G)=\{v_0\}$.

\begin{lemma}
\label{lemma: forests can have identity gains}
    Let $(G,\psi)$ be a $\Gamma$-gain graph. For any forest $T$ in $E(G)$, there is some $\psi'$ equivalent to $\psi$ such that $\psi'(e)=\textrm{id}$ for all $e\in T$.
\end{lemma}

Notice that $p(v)$ is the zero vector if  $v\in V_0(G)$. When drawing the $\Gamma$-gain graph of $\tilde{G}$, we use a black circle to denote the fixed vertex, and white circles to denote the elements of $\overline{V(G)}$ (see Figure~\ref{fig: first gain graph example}).

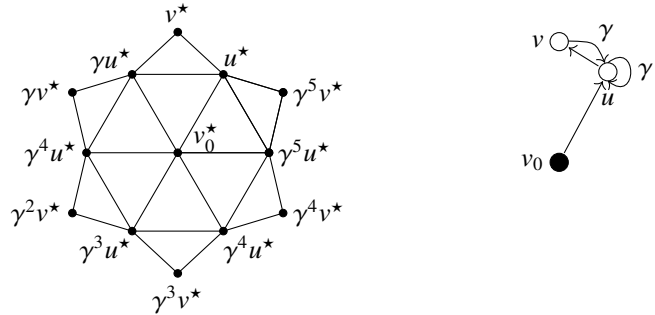
\begin{figure}[H]
    \centering
         \begin{tikzpicture}[scale = 0.8] 
        \foreach \x in {0,60,120,...,360} \draw (\x:1.5cm) -- (\x+60:1.5cm);
        \foreach \x in {0,60,120,...,360} \draw (\x:1.5cm) -- (\x+30:2cm) -- (\x+60:1.5cm);
      \foreach \x in {0,60,120,...,360} \draw (\x:0cm) -- (\x:1.5cm);
     \foreach \x in {0,60,120,...,360} \node[inner sep=1pt,circle,draw,fill] at (\x:1.5cm) {};
     \foreach \x in {0,60,120,...,360} \node[inner sep=1pt,circle,draw,fill] at (\x+30:2cm) {};
     \node[inner sep=1pt,circle,draw,fill] at (0:0cm) {};
     \node[right] at (0.1,0.25){$v_0^{\star}$};
     \node[right] at (0.6,1.6){$u^{\star}$};
     \node[left] at (-0.6,1.53){$\gamma u^{\star}$};
     \node[left] at (-1.5,0){$\gamma^4u^{\star}$};
     \node[left] at (-0.6,-1.6){$\gamma^3u^{\star}$};
     \node[right] at (0.6,-1.6){$\gamma^4u^{\star}$};
     \node[right] at (1.5,0){$\gamma^5u^{\star}$};
     \node[above] at (0,2){$v^{\star}$};
     \node[left] at (-1.75,1){$\gamma v^{\star}$};
     \node[left] at (-1.75,-1){$\gamma^2v^{\star}$};
     \node[below] at (0,-2.05){$\gamma^3v^{\star}$};
     \node[right] at (1.75,-1){$\gamma^4v^{\star}$};
     \node[right] at (1.75,1){$\gamma^5v^{\star}$};
      \end{tikzpicture}\hspace{2cm}
    \begin{subfigure}{0.2\textwidth}
        \begin{tikzpicture}[scale = 0.8] 
            \node at (0,-2.5){};
            \node (v_0) at (0,0){};
            \node (u) at (0.8,1.5){};
            \node (v) at (0,2){};
            \draw[->] (v_0) -- (u);
            \draw[->] (u) -- (v);
            
            \draw (u) circle (0.15cm);
            \draw (v) circle (0.15cm);
            \draw[->] (0.15, 2.01) .. controls (0.6, 2) .. (0.75, 1.65);
            \draw[->] (0.8,1.65) .. controls (1.3,2.05) and (1.3,0.95) .. (0.8,1.35);
            \filldraw[black] (v_0) circle (0.15cm);
            \node[right] at (1.15, 1.5){$\gamma$};
            \node[right] at (0.5,2.15){$\gamma$};
            \node[below] at (0.8, 1.3){$u$};
            \node[left] at (-0.1,0) {$v_0$};
            \node[left] at (-0.1,2) {$v$};
        \end{tikzpicture}
    \end{subfigure}
    \caption{A $\Gamma$-symmetric graph and its $\Gamma$-gain graph. Here, $\Gamma\simeq\mathbb{Z}_6$ through an isomorphism which sends $\gamma\in\Gamma$ to 1. The unlabelled edges have gain $\text{id}$.}
    \label{fig: first gain graph example}
\end{figure}

This process gives rise to the class of $\Gamma$-gain graphs (see Definition~\ref{def:gaingraph} below), and it may be reversed in order to construct a unique $\Gamma$-symmetric graph $\tilde G$ from a $\Gamma$-gain graph $(G,\psi)$ (see Section 3.2 in \cite{SchulzeLP2024} for details).  

\begin{definition} \label{def:gaingraph}
  Let $\Gamma$ be a cyclic group. A \emph{$\Gamma$-gain graph} is a pair $(G,\psi),$ where $G$ is a directed multigraph and $\psi:E(G)\rightarrow\Gamma$ is a function that assigns a label to each edge such that, for some partition $V(G)=V_0(G)\,\dot\cup\,\overline{V(G)}$, the set $V_0(G)$ has at most one vertex, and the following conditions are satisfied:

\begin{enumerate}
\item if $e,f\in E(G)$ are parallel and have the same direction, then $\psi(e)\neq\psi(f)$. If they are parallel and have opposite directions, then $\psi(e)\neq\psi(f)^{-1}$;
\item if $V_0(G)$ contains a vertex $v_0$, then $v_0$ is not incident to a loop or parallel edges;
\item if $e\in E(G)$ is a loop, $\psi(e)\neq\textrm{id}$.
\end{enumerate} 
We call $\psi$ the \textit{gain function} of $(G,\psi)$. The elements of $\overline{V(G)}$ are called the \emph{free vertices} of $(G,\psi)$ and, if $V_0(G)\neq\emptyset$, the only element in $V_0(G)$ is called the \textit{fixed vertex of }$(G,\psi)$, and is usually denoted $v_0$.
\end{definition}

Let $k:=|\Gamma|\geq4.$ It follows from the main result of this paper, that the combinatorics of a $\Gamma$-gain graph determines whether a $\mathcal{C}_k$-generic realisation of its $\Gamma$-lifting is infinitesimally rigid, for $k=4,6$, or odd $k$ less than 1000 (see Theorem~\ref{final theorem for rotation}). In Section~\ref{sec:gaincounts}, we describe the combinatorial counts that the $\Gamma$-gain graph must satisfy in order to obtain infinitesimal rigidity. First, we need to define the notions of balancedness, near-balancedness and $S(k,j)$. All such notions may be found in Section 4.1 of \cite{Ikeshita} and in Section 1 of \cite{RIandST16} (for the case where $V_0(G)=\emptyset$), and the notion of near-balancedness can be found in Section 2.2 of \cite{jzt2016}. Furthermore, the same notions can be found in Section 3.3 of \cite{SchulzeLP2024} also for the case where $V_0(G)\neq\emptyset$.

\subsection{Balanced, near-balanced and S(k,j) gain graphs}
Let $(G,\psi)$ be a connected $\Gamma$-gain graph and let $W=v_1e_1v_2,\dots,v_{t-1}e_tv_t$ be a walk in $(G,\psi)$. The \textit{gain of $W$} is $\psi(W)=\prod_{i=1}^t\psi(e_i)^{\textrm{sign}(e_i)}$, where $\textrm{sign}(e_i)=1$ if $e_i$ is directed from $v_i$ to $v_{i+1}$, and $\textrm{sign}(e_i)=-1$ otherwise. We use $\left<E(G)\right>$ (or $\left<G\right>$) to denote the group generated by $\{\psi(W):W\text{ is a closed walk in }G\text{ with no fixed vertex}\}$.
Given $0\leq m\leq 2,0\leq l\leq 3$, we say a $(G,\psi)$ is \textit{$(2,m,l)$-sparse} if $|E(H)|\leq 2|\overline{V(H)}|+m|V_0(H)|-l$ for all subgraphs $H$ of $G$ with $E(H)\neq\emptyset$, and we say $(G,\psi)$ is \textit{$(2,m,l)$-tight} if it is $(2,m,l)$-sparse and $|E(G)|=2|\overline{V(G)}|+m|V_0(G)|-l$. We abbreviate $(2,2,l)$-sparse and $(2,2,l)$-tight to $(2,l)$-sparse and $(2,l)$-tight, respectively. 

\subsubsection{Balancedness}
We say a $\Gamma$-gain graph $(G,\psi)$ (equivalently, $G,E(G)$) is \textit{balanced} if $\left<G\right>=\{\text{id}\}$. Otherwise, we say $(G,\psi)$ (equivalently, $G,E(G)$) is \textit{unbalanced}. Lemma 2.4 in \cite{jzt2016} states that $(G,\psi)$ is balanced if and only if it has an equivalent $\Gamma$-gain graph $(G,\psi')$ with $\psi'(e)=\text{id}$ for all $e\in E(G)$. The following result is proved in Section 4 of \cite{Ikeshita}, in the case when $V_0(G)=\emptyset.$ It is straightforward to see that the same arguments can be used to show that the results still hold when $V_0(G)\neq\emptyset.$ See \cite{PhDThesis} for details.  

\begin{proposition}
    \label{union of balanced and [...] is [...]}
    Let $(G,\psi)$ be a $\Gamma$-gain graph and $H_1,H_2$ be connected subgraphs of $G$, such that $H_1\cap H_2$ is connected, with no fixed cut-vertex. Let $H:=H_1\,\cup\,H_2$. The following hold:
    \begin{itemize}
        \item[(i)] If $H_1$ is balanced, then $\left<H\right>=\left<H_2\right>$. In particular, if $H_2$ is balanced, then so is $H$.
        \item[(ii)] If $\left<H_1\right>\simeq\mathbb{Z}_p$ for some prime $p$, and $H_1\cap H_2$ is unbalanced, then $\left<H_1\right>\simeq\left<H_1\cap H_2\right>$ and $\left<H_2\right>\simeq\left<H\right>$.
    \end{itemize}
\end{proposition}

\subsubsection{Near-balancedness}
Let $(G,p)$ be a $\Gamma$-gain graph with $V_0(G)=\emptyset$, and let $v\in V(G),\delta\in\Gamma$. We say $(G,p)$ (equivalently, $G,E(G)$) is \textit{near-balanced} with \textit{base vertex} $v$ and \textit{gain} $\delta$ if it is unbalanced, and all closed walks $W$ starting at $v$ and not containing $v$ as an internal vertex have gain $\textrm{id},\delta$ or $\delta^{-1}$. If $\left<G\right>\simeq\mathbb{Z}_2,\mathbb{Z}_3$, then $G$ is trivially near-balanced. Hence, we say that $G$ is \textit{proper near-balanced} if it is near-balanced and $\left<G\right>\not\simeq\mathbb{Z}_2,\mathbb{Z}_3$. Lemma 4.1 in \cite{Ikeshita} shows that $(G,\psi)$ is proper near-balanced if and only if it is unbalanced, and there is some $\delta\in\Gamma$ and some $\Gamma$-gain graph $(G,\psi')$ equivalent to $(G,\psi)$, such that $\psi'(e)\in\{\text{id},\delta\}$ for all edges $e\in E(G)$ directed to $v$, and $\psi'(e)=\text{id}$ for all edges $e\in E(G)$ not incident to $v$.  If a $\Gamma$-gain graph $(G,\psi)$ is near-balanced and not $(2,1)$-sparse, then it is easy to see that the rigidity matrix of any $\mathcal{C}_k$-generic realisation of its $\Gamma$-lifting is row dependent 
(see Figure~\ref{fig: examples of S_0(k,j)}(a,b)).
The following statement (and its proof) are slight modifications of Lemmas 4.4, 4.5, 4.6 and 4.10 in \cite{Ikeshita} (for details, see \cite{PhDThesis}).


\begin{proposition}
\label{union of [...] and [...] is [...]}
   Let $(G,\psi)$ be a $\Gamma$-gain graph and $H_1,H_2$ be subgraphs of $G$ with $V_0(H_1)=V_0(H_2)=\emptyset$, and let $H:=H_1\cup H_2$. The following hold:
       \begin{itemize}
        \item[(i)] Suppose that, for $1\leq i\leq 2$, $H_i$ is proper near-balanced, $|E(H_i)|=2|V(H_i)|$,       
         and there is an edge $f_i\in E(H_i)$ such that $H_i-f_i$ is $(2,1)$-tight. If $H_1\cap H_2$ is proper near-balanced and $(2,1)$-tight, then $H$ is proper near-balanced.
       \item[(ii)]  Suppose that $H_1$ is proper near-balanced, $|E(H_1)|=2|V(H_1)|$, and there is an edge $f_1\in E(H_1)$ such that $H_1-f_1$ is $(2,1)$-tight. If $H_2$ is connected and balanced, and $H_1\cap H_2$ is connected, balanced and $(2,3)$-tight, then $H$ is proper near-balanced.
       \item[(iii)] Suppose $H_1$ is balanced, $|E(H_1)|=2|V(H_1)|-2$, and there is an edge $f_1\in E(H_1)$ such that $H_1-f_1$ is $(2,3)$-tight. If $H_2$ is connected and balanced, and $H_1\cap H_2$ consists of two connected components, one of which is an isolated vertex $v$, then $H_1\cup H_2$ is near-balanced with base vertex $v$.
       \item[(iv)] If $H_1,H_2$ are connected, $H_1$ is proper near-balanced, $\left<H_2\right>\simeq\mathbb{Z}_p$ for some prime $p$, and $H_1\cap H_2$ is connected and unbalanced, then $\left<H_1\right>\simeq\left<H_2\right>\simeq\left<H_1\cap H_2\right>\simeq\left<H\right>$. 
    \end{itemize}
\end{proposition}

\subsubsection{S(k,j) gain graphs}
Now, let $k\geq4$, and for $0\leq j\leq k-1,-1\leq i\leq 1$, define the following sets:
\begin{equation*}
    S_i(k,j)=\begin{cases}
        \{n\in\mathbb{N}:n\geq2,n|k,j\equiv i(\bmod n)\} & \text{if }j\text{ is even}\\
        \{n\in\mathbb{N}:n\geq2,n|k,n\neq2,j\equiv i(\bmod n)\} & \text{if }j\text{ is odd}
    \end{cases}
\end{equation*}
Let $(G,\psi)$ be a $\mathbb{Z}_k$-gain graph. If $\left<G\right>\simeq\mathbb{Z}_n$ for some $n\in S_0(k,j)$, we say $G$ (equivalently, $G,E(G)$) is $S_0(k,j)$. If $\left<G\right>\simeq\mathbb{Z}_n$ for some $n\in S_{-1}(k,j)\cup S_1(k,j)$, then we say $G$ (equivalently, $G,E(G)$) is $S_{\pm1}(k,j)$. If $G$ is either $S_0(k,j)$ or $S_{\pm1}(k,j)$, we say $G$ (equivalently, $G,E(G)$) is $S(k,j)$.

If $\left<G\right>\simeq\mathbb{Z}_n$ for some $2\leq n\leq k-1$, then the $\mathbb{Z}_k$-lifting $\tilde G$ of $(G,\psi)$ is a $\mathbb{Z}_n$-symmetric graph. Let $(\tilde G,\tilde p)$ be a $\mathcal{C}_n$-generic realisation of $\tilde G$. If $n\in S_0(k,j)$, then a $\rho_j$-symmetric infinitesimal motion of $(\tilde G,\tilde p)$ (as a $\mathcal{C}_k$-symmetric framework) is also a $\rho_0$-symmetric infinitesimal motion of $(\tilde G,\tilde p)$ (as a $\mathcal{C}_n$-symmetric framework). See Example~\ref{ex:S(k,j)} for an instance with $k=9,j=3,n=3$. Similarly, if $n\in S_{-1}(k,j)\cup S_1(k,j)$, then a $\rho_j$-symmetric infinitesimal motion of $(\tilde G,\tilde p)$ (as a $\mathcal{C}_k$-symmetric framework) is a $\rho_1$-symmetric  or a $\rho_{k-1}$-symmetric infinitesimal motion of $(\tilde G,\tilde p)$ (as a $\mathcal{C}_n$-symmetric framework). 

\begin{example}
\label{ex:S(k,j)}
Let $\Gamma=\left<\gamma\right>$ be a cyclic group of order 9, and recall that $\Gamma\simeq\mathbb{Z}_9$ through the isomorphism which maps $\gamma$ to 1. Let $(G,\psi)$ be the $\Gamma$-gain graph whose vertex set is $V(G)=\{u,v_0\}$, where $u$ is free and $v_0$ is fixed, whose edge set is $\{e=(u,u),f=(u,v_0)\}$, and whose gain function is defined by letting $\psi(e)=\gamma^3,\psi(f)=\text{id}$ (see Figure~\ref{fig: examples of S_0(k,j)}(c)). Then, the $\Gamma$-lifting $\tilde G$ of $(G,\psi)$ is also symmetric with respect to the subgroup $\mathbb{Z}_3$ of $\mathbb{Z}_9$. Take a $\mathcal{C}_3$-generic realisation $(\tilde G,\tilde p)$ of $\tilde G$. Let $m$ be a $\rho_3$-symmetric infinitesimal motion of $(\tilde G,\tilde p)$ (when seen as a $\mathcal{C}_9$-symmetric framework), and let $\Gamma'=\left<\gamma'\right>\simeq\mathbb{Z}_3$ be such that $\gamma'$ is mapped to 1 through an isomorphism. Then $m$ satisfies $m(\gamma' v^{\star})=m(\gamma^3v^{\star})=\overline{\omega_9}^{3\cdot3}C_9^3m(v^{\star})=C_3m(v^{\star})$. Hence, $m$ is a $\rho_0$-symmetric infinitesimal motion of the $\mathcal{C}_3$-symmetric framework $(\tilde G,\tilde p)$ (see Figure~\ref{fig: examples of S_0(k,j)}(d)).  
\end{example}

The proof of Lemma 2.2 in \cite{RIandST16} shows that, in the case where $V_0(G)=\emptyset$, the union of $S_i(k,j)$ graphs is also $S_i(k,j)$ under suitable conditions (see also Lemmas 4.19 and 4.20 in \cite{Ikeshita}). It is straightforward to generalise this argument to show that statements (i) and (ii) in Proposition~\ref{gcd proposition} hold. For the third statement in Proposition~\ref{gcd proposition}, we can use a very similar argument to the one used for the proof of statement (iii) in Proposition~\ref{union of [...] and [...] is [...]}. We refer the reader to \cite{PhDThesis} for details.

\begin{figure}[H]
    \centering
    \begin{tikzpicture}[scale = 0.7]
        \draw(-0.5,0) circle (0.15cm);
        \draw(2.5,0) circle (0.15cm);
        \draw(1,-0.3) circle (0.15cm);
        \draw(1,2) circle (0.15cm);
  \draw[->] (1.15,2) .. controls (1.55,2.5) and (0.45,2.5) .. (0.85,2);
  \draw[->](1,1.85) -- (1,-0.15);
  \draw[->] (0.85,-0.3) -- (-0.35,0);
  \draw[->] (1.15,-0.3) -- (2.35,0);
  \draw[->] (2.4,0.1) -- (1.1,1.9);
  \draw[->] (1.15,1.95) .. controls (2,1.3) .. (2.5,0.15);
  \draw[->] (-0.4,0.1) -- (0.9,1.9);
  \draw[->] (0.85,1.95) .. controls (0,1.3) .. (-0.5,0.15);
  \node at (-0.1,1.3) {$\gamma$};
  \node at (2.1,1.3) {$\gamma$};
  \node at (1,2.6) {$\gamma$};
  \node[below] at (1,-1.85) {(a)};
    \end{tikzpicture}\hspace{5mm}
    \begin{subfigure}[b]{.2\textwidth}
   \centering
      \begin{tikzpicture} 
        \foreach \x in {60,120,180,...,360} \draw (\x:0.75cm) -- (\x+60:0.75cm);
        \foreach \x in {0,60,120,...,360} \draw (\x:0.75cm) -- (\x+30:1cm) -- (\x+60:0.75cm);
        \foreach \x in {0,60,120,...,360} \draw (\x:0.75cm) -- (\x+30:1.375cm) -- (\x+60:0.75cm);
        \foreach \x in {0,60,120,...,360} \draw (\x:0.75cm) -- (\x+10: 1.25cm) -- (\x+30: 1.375cm);
        \foreach \x in {60,120,...,360} \draw (\x+10:1.25cm) -- (\x+30:1.375cm);
        \foreach \x in {60,120,...,360} \draw (\x+30:1cm) -- (\x+10: 1.25cm);
        \foreach \x in {60,120,...,360} \node[inner sep=1pt,circle,draw,fill] at (\x:0.75cm) {};
     \foreach \x in {60,120,...,360} \node[inner sep=1pt,circle,draw,fill] at (\x+30:1cm) {};
     \foreach \x in {60,120,...,360} \node[inner sep=1pt,circle,draw,fill] at (\x+30:1.375cm) {};
     \foreach \x in {60,120,...,360} \node[inner sep=1pt,circle,draw,fill] at (\x+10:1.25cm) {};
     \node[below] at (0,-2.3) {(b)};
      \end{tikzpicture}
\end{subfigure}\hspace{4mm}
\begin{subfigure}{.15\textwidth}
    \begin{tikzpicture}[scale = 0.75]
        \draw(-0.5,1.8) circle (0.15cm);
        \draw[fill=black](1.5,1) circle (0.15cm);
        \draw[->] (-0.35,1.8) .. controls (0.05,2.3) and (-1.05,2.3) .. (-0.65,1.8);
  \draw[->] (-0.33, 1.8) -- (1.35, 1.1);
  \node[above] at (-0.5, 2.1){$\gamma^3$};
  \node[below] at (0.25,-1.85) {(c)};
    \end{tikzpicture}
\end{subfigure}
    \begin{subfigure}[b]{.25\textwidth}
        \begin{tikzpicture} 
        \newdimen\R
   \R=1.5cm
\foreach \x in {40,80,...,360} \draw (0:0) -- (\x:\R);
\foreach \x in {40,160,280} \draw (\x:\R) -- (\x+120:\R);
\foreach \x in {80,200,320} \draw (\x:\R) -- (\x+120:\R);
\foreach \x in {120,240,360} \draw (\x:\R) -- (\x+120:\R);
\foreach \x in {40,80,...,360} \node[inner sep=1pt,circle,draw,fill] at (\x:\R) {};
\node at (160:1.8cm){$v^{\star}$};
\node at (200:1.9cm){$\gamma v^{\star}$};
\node at (240:1.8cm){$\gamma^2v^{\star}$};
\node[inner sep=1pt,circle,draw,fill] at (0:0) {};
     \node[below] at (0,-2.3) {(d)};
      \end{tikzpicture}
    \end{subfigure}
    \caption{(a) is a proper near-balanced $\Gamma$-gain graph with $\Gamma$-lifting (b). (c) is a $S_0(9,j)$ $\Gamma$-gain graph, where $|\Gamma|=9$, and (d) is its $\Gamma$-lifting. In (a,b), the unlabelled edges have gain $\text{id}$.}
    \label{fig: examples of S_0(k,j)}
\end{figure}
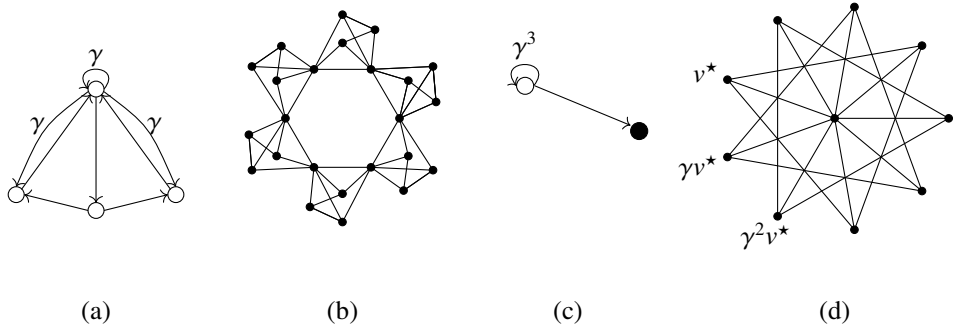

\begin{proposition}
\label{gcd proposition}
Let $k:=|\Gamma|\geq4$, and $(G,\psi)$ be a $\Gamma$-gain graph. Let $H_1,H_2$ be subgraphs of $G$ such that $\left<H_1\right>\simeq\mathbb{Z}_n$ and $\left<H_2\right>\simeq\mathbb{Z}_m$ for some positive integers $n,m$ that divide $k$. Let $g=\textrm{gcd}(n,m)$ and $l=\textrm{lcm}(n,m)$. Assume that, for some $i\in\{-1,0,1\}$, $n\in S_i(k,j)$. The following hold:
\begin{itemize}
    \item[(i)] If $g\neq1$ and $m\in S_{i'}(k,j)$ for some $i'\in\{-1,0,1\}$, then $i=i'$.
    \item[(ii)] Suppose $H_1\cap H_2$ is connected, or it is composed of two connected components, one of which is the isolated fixed vertex. If $n,m\in S_i(k,j)$, then $H$ is $S_i(k,j)$. 
    \item[(iii)] If $H_2$ is near-balanced and $H_1\cap H_2$ is connected unbalanced, then $H_1\cup H_2$ is $S_i(k,j)$.
\end{itemize}
\end{proposition}


\subsection{Gain sparsity of a gain graph}\label{sec:gaincounts}
We now introduce the combinatorial counts which characterise $\mathcal{C}_k$-generic infinitesimally rigid  frameworks. They are dependent on the notion of balancedness. For $2\leq j\leq k-2$, they are also dependent on the notions of near-balancedness and $S(k,j)$. 

\begin{definition}
\label{sparsity defn.}
    Let $(G,\psi)$ be a $\Gamma$-gain graph. Let $0\leq m\leq1,1\leq l \leq2$. We say $\left(G,\psi\right)$ is \emph{$(2,m,3,l)$-gain-sparse} if it is $(2,m,l)$-sparse and all of its balanced subgraphs (with non-empty edge set) are $(2,3)$-sparse. We say $(G,\psi)$ is \textit{$(2,m,3,l)$-gain-tight} if it is $(2,m,3,l)$-gain-sparse and $(2,m,l)$-tight.
\end{definition}

Let $k:=|\Gamma|\geq4$, and $(\tilde G,\tilde p)$ be a $\mathcal{C}_k$-generic framework, whose underlying graph has $\Gamma$-gain graph $(G,\psi)$. Theorem 7.13 in \cite{SchulzeLP2024} states that $(\tilde G,\tilde p)$ is $\rho_0$-symmetrically isostatic if and only if $(G,\psi)$ is $(2,0,3,1)$-gain tight, and that it is $\rho_1$-symmetrically isostatic and $\rho_{k-1}$-symmetrically isostatic if and only if $(G,\psi)$ is $(2,1,3,2)$-gain tight. However, when $2\leq j\leq k-2$, $(G,\psi)$ must satisfy more refined conditions, in order for $(\tilde G,\tilde p)$ to be $\rho_j$-symmetrically isostatic. Hence, we define the following.


\begin{definition}
\label{definition of sparsity, higher order rotation}
    Let $k:=|\Gamma|\geq4$, and $2\leq j\leq k-2$, let $(G,\psi)$ be a $\Gamma$-gain graph, and let $F\subseteq E(G)$. Let $C(F)$ denote the set of connected components of $F$. We define the map $\alpha_k^j:C(F)\rightarrow\{0,1,2,3\}$ by letting  
    \begin{equation*}
    \alpha_k^j(X)=\begin{cases}
        0 & \text{if }X\text{ is balanced}\\
        1 & \text{if }j\text{ is odd and }\left<X\right>\simeq\mathbb{Z}_2\\
        2-|V_0(X)| & \text{if }X\text{ is }S_{\pm1}(k,j)\\
        2-2|V_0(X)| & \text{if }X\text{ is }S_0(k,j)\text{ or } |V_0(X)|=0 \text{ and } X \text { is proper near-balanced}\\
        3-2|V_0(X)| & \text{otherwise}
    \end{cases}
\end{equation*} 
    Since the notion of near-balancedness is only defined for graphs with no fixed vertex, if $X$ is proper near-balanced, then $\alpha_k^j(X)=2$.
    
    We also define the function $f_k^j$ on $2^{E(G)}$ by letting $f_k^j(F)=\sum_{X\in C(F)}\left\{2|V(X)|-3+\alpha_k^j(X)\right\}.$ We say $(G,\psi)$ is \textit{$\mathbb{Z}^{j}_k$-gain sparse} if $|E(H)|\leq f_k^j(E(H))$ for all subgraphs $H$ of $G$ with non-empty edge set. We say $(G,\psi)$ is \textit{$\mathbb{Z}^{j}_k$-gain tight} if it is $\mathbb{Z}^{j}_k$-gain sparse and $|E(G)|=f_k^j(E(G))$.
\end{definition}

\begin{remark}
\label{the count is matroidal}
    By Lemma 4.19(d) in \cite{Ikeshita}, $\alpha_k^j$, and hence $f_k^j$, are well-defined. Moreover, it will follow from one of the main results of this paper (Theorem ~\ref{theorem: final theorem for k-fold symmetry pt.2}) that the count in Definition~\ref{definition of sparsity, higher order rotation} is matroidal if $4\leq k<1000$ is odd or $k=4,6$. We expect that the same is true for all other symmetry groups in the plane using a proof that is analogous to the one given in \cite{Ikeshita}.
\end{remark} 

In \cite{SchulzeLP2024}, we showed that, if $2\leq j\leq k-2,$ then this count is necessary to have a $\rho_j$-symmetrically isostatic framework (see Proposition~\ref{necessary conditions for higher rotation} below). We aim to show that it is also sufficient.

\begin{proposition}[Proposition 5.4 in \cite{SchulzeLP2024}]
\label{necessary conditions for higher rotation}
For $k:=|\Gamma|\geq4$, let $(\Tilde{G},\Tilde{p})$ be a $\mathcal{C}_k$-symmetric framework, and let $(G,\psi,p)$ be the $\Gamma$-gain framework of $(\tilde{G},\tilde p)$. Let $2\leq j\leq k-2$. If $(\Tilde{G},\Tilde{p})$ is $\rho_j$-symmetrically isostatic, then $(G,\psi)$ is $\mathbb{Z}_k^j$-gain tight.
\end{proposition}

\section{Gain graph extensions}\label{sec:red}
The main result of this paper relies on a proof by induction on the order of a $\Gamma$-gain graph. Hence, we introduce some moves, known as \textit{extensions} (and their counterparts, known as \textit{reductions}), which add one or more free vertices to a $\Gamma$-gain graph. Throughout this section, we let $(G,\psi)$ be a $\Gamma$-gain graph, and we will construct a $\Gamma$-gain graph $(G',\psi')$ by applying an extension to $(G,\psi)$. 

\begin{definition}
\label{defn: 0-extension}
A \textit{0-extension} chooses two vertices $v_1,v_2\in V(G)$ ($v_1,v_2$ may coincide, provided they are free) and adds a free vertex $v$, together with two edges $e_1=(v,v_1),e_2=(v,v_2)$. We let $\psi'(e)=\psi(e)$ for all $e\in E(G)$, and label the new edges freely, provided $v_1,v_2$ do not coincide. If $v_1,v_2$ coincide, we choose $\psi'$ such that $\psi'(e_1)\neq\psi'(e_2)$. 
\end{definition}

\begin{definition}
\label{defn: loop-1-extension}
A \textit{loop-1-extension} chooses a vertex $u\in V(G)$ and adds a free vertex $v$ to $V(G)$, as well as an edge $e=(v,u)$ and a loop $f=(v,v)$. We let $\psi'(e')=\psi(e')$ for all $e'\in E(G)$, we label $e,f$ freely, with the condition that $\psi'(f)\neq\textrm{id}$. 
\end{definition}

\begin{definition}
\label{defn: 1-extension}
A \textit{1-extension} chooses a vertex $u\in V(G)$ and an edge $e=\left(v_1,v_2\right)\in E(G)$ ($v_1,v_2,u$ may coincide, provided they are free; any pair of free vertices in $\{v_1,v_2,u\}$ may coincide). It removes $e$ and adds a free vertex $v$ to $V(G)$, as well as the edges $e_1=(v,v_1),e_2=(v,v_2),e_3=(v,u)$. We let $\psi'(f)=\psi(f)$ for all $f\in V(G)$, we label $e_1,e_2$ such that $\psi'(e_1)^{-1}\psi'(e_2)=\psi(e)$, and $e_3$ is labelled such that, if there is a 2-cycle $e':=e_3e_i$ for some $1\leq i\leq2$, then $\psi'(e')\neq\textrm{id}$. 
\end{definition}

The following move may only be applied to a $\Gamma$-gain graph $(G,\psi)$ such that $|\Gamma|$ is even and $V_0(G)=\{v_0\}.$ Recall that $\Gamma=\left<\gamma\right>$ is isomorphic to $\mathbb{Z}_k$ through the isomorphism which maps 1 to $\gamma$. 

\begin{definition}
    A \textit{2-vertex-extension} adds two free vertices $v_1,v_2$ to $V(G)$, as well as the edges $e_1=(v_1,v),e_2=(v_2,v),f_1=(v_1,v_2)$ and $f_2=(v_2,v_1).$ We let $\psi'(e)=\psi(e)$ for all $e\in E(G)$, we label $f_1,f_2$ with $\textrm{id}$ and $\gamma^{k/2}$, respectively, and $f_1,f_2$ are labelled freely.
\end{definition}

The inverse operations of a 0-extension, loop-1-extension, 1-extension and 2-vertex-extension are called \textit{0-reduction}, \textit{loop-1-reduction}, \textit{1-reduction} and \textit{2-vertex-reduction}, respectively. Figure~\ref{extensions image} gives an illustration of each extension, together with the corresponding reduction.

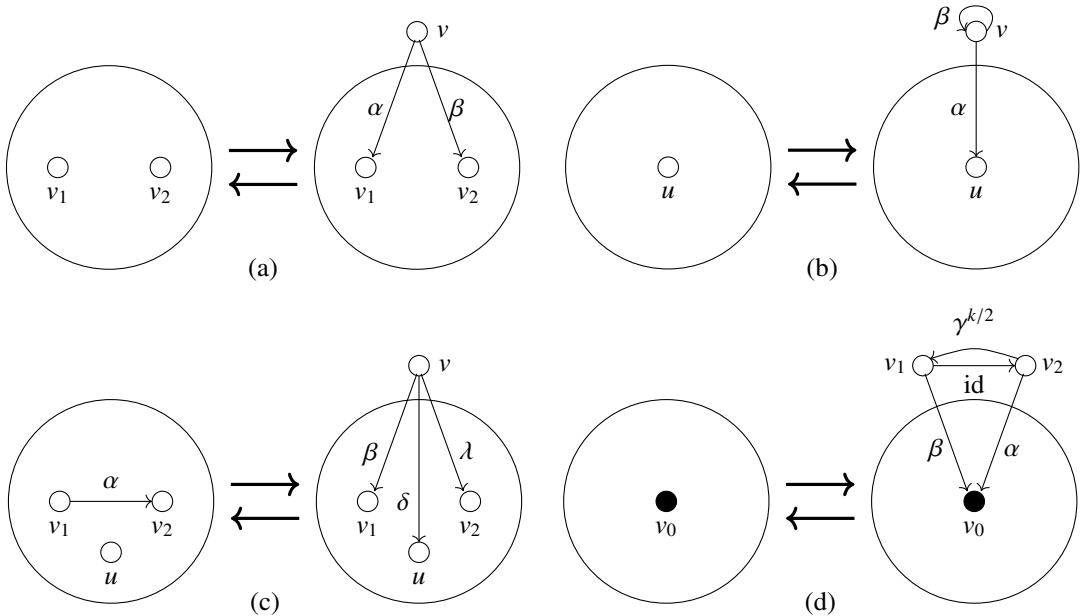
\begin{figure}[H]
    \centering
    \begin{tikzpicture}[scale = 0.9]
    \node at (5, -1.5){(a)};
  \draw (2.75,0) circle (1.5cm);
  \draw(2,0) circle (0.15cm);
  \draw(3.5,0) circle (0.15cm);
  \node[below] at (2,-0.15){$v_1$};
  \node[below] at (3.5,-0.15){$v_2$};

  \draw[->, very thick] (4.5,0.25) -- (5.5,0.25);
  \draw[->, very thick] (5.5,-0.25) -- (4.5,-0.25);

  \draw(6.5,0) circle (0.15cm);
  \draw(8,0) circle (0.15cm);
  \draw (7.25,0) circle (1.5cm);
  \draw (7.25,2) circle (0.15cm);
  \node[below] at (6.5,-0.15){$v_1$};
  \node[below] at (8,-0.15){$v_2$};
  \node[right] at (7.4, 2){$v$};
  \draw[->] (7.28, 1.88) -- (7.9, 0.15);
  \draw[->] (7.22, 1.88) -- (6.6, 0.15);
  \node[left] at (6.9,0.85) {$\alpha$};
  \node[right] at (7.6,0.85) {$\beta$};
\end{tikzpicture}\hspace{5mm}
    \begin{subfigure}[b]{0.45\textwidth}
        \begin{tikzpicture}[scale = 0.9]
        \node at (5, -1.5){(b)};
  \draw (2.75,0) circle (1.5cm);
  \draw(2.75,0) circle (0.15cm);
  \node[below] at (2.75,-0.15){$u$};

  \draw[->, very thick] (4.5,0.25) -- (5.5,0.25);
  \draw[->, very thick] (5.5,-0.25) -- (4.5,-0.25);

  \draw(7.25,0) circle (0.15cm);  
  \draw (7.25,0) circle (1.5cm);
  \draw (7.25,2) circle (0.15cm);
  \node[below] at (7.25,-0.15){$u$};
  \node[right] at (7.4, 2){$v$};
  \draw[->] (7.25, 1.85) -- (7.25, 0.15);
  \node[left] at (7.25,0.85) {$\alpha$};
  \draw[->] (7.4,2) .. controls (7.8,2.5) and (6.7,2.5) .. (7.1,2); 
  \node[left] at (7,2.2) {$\beta$};
\end{tikzpicture}
    \end{subfigure}\vspace{1mm}
    \begin{subfigure}[b]{0.45\textwidth}
        \begin{tikzpicture}[scale = 0.9]
        \node at (5, -1.5){(c)};
  \draw (2.75,0) circle (1.5cm);
  \draw(2,0) circle (0.15cm);
  \draw(3.5,0) circle (0.15cm);
  \draw(2.75,-0.75) circle (0.15cm);
  \draw[->](2.15,0) -- (3.35,0);
  \node[below] at (3.5,-0.15){$v_2$};
  \node[below] at (2,-0.15){$v_1$};
  \node[below] at (2.75,-0.9){$u$};
  \node[above] at (2.75, 0.05){$\alpha$};

  \draw[->, very thick] (4.5,0.25) -- (5.5,0.25);
  \draw[->, very thick] (5.5,-0.25) -- (4.5,-0.25);

  \draw(6.5,0) circle (0.15cm);  
  \draw (7.25,0) circle (1.5cm);
  \draw (8,0) circle (0.15cm);
  \draw (7.25, 2) circle (0.15cm);
  \draw(7.25,-0.75) circle (0.15cm);
  \node[right] at (7.4, 2){$v$};
  \node[below] at (6.5,-0.15){$v_1$};
  \node[below] at (8, -0.15){$v_2$};
  \node[below] at (7.25,-0.9){$u$};
  \draw[->] (7.28, 1.88) -- (7.9, 0.15);
  \draw[->] (7.22, 1.88) -- (6.6, 0.15);
  \draw[->] (7.25,1.85) -- (7.25,-0.6);
  \node[right] at (7.7,0.75) {$\lambda$};
  \node at (6.55,0.7) {$\beta$};
  \node[left] at (7.25,0) {$\delta$};
\end{tikzpicture}
    \end{subfigure}\hspace{5mm}
    \begin{subfigure}[b]{0.45\textwidth}
        \begin{tikzpicture}[scale = 0.9]
        \node at (5, -1.5){(d)};
  \draw (2.75,0) circle (1.5cm);
  \draw[fill = black](2.75,0) circle (0.15cm);
  \node[below] at (2.75,-0.15){$v_0$};

  \draw[->, very thick] (4.5,0.25) -- (5.5,0.25);
  \draw[->, very thick] (5.5,-0.25) -- (4.5,-0.25);

  \draw[fill = black](7.25,0) circle (0.15cm);  
  \draw (7.25,0) circle (1.5cm);
  \draw (6.5,2) circle (0.15cm);
  \draw (8,2) circle (0.15cm);
  \node[below] at (7.25,-0.15){$v_0$};
  \draw[->] (6.53, 1.88) -- (7.15, 0.15);
  \draw[->] (7.97, 1.88) -- (7.35, 0.15);
  \node[right] at (7.55,0.75) {$\alpha$};
  \node[left] at (6.93,0.75) {$\beta$};
  \draw[->] (6.65,2) -- (7.85,2);
  \draw[->](7.9,2.1) .. controls (7.25,2.3) ..  (6.6,2.1);
  \node[left] at (6.4,2){$v_1$};
  \node[right] at (8.1,2){$v_2$};
  \node[above] at (7.25,2.3){$\gamma^{k/2}$};
  \node[below] at (7.25,2){$\textrm{id}$};
\end{tikzpicture}
    \end{subfigure}
        \caption{Examples of extensions. (a) is a 0-extension, where the gains $\alpha$ and $\beta$ are arbitrary. (b) is a loop-1-extension, where $\alpha\neq\textrm{id}$ and $\beta$ is a arbitrary. (c) is a 1-extension, where $\alpha=\beta\lambda^{-1}$ and $\delta$ is arbitrary. (d) is a 2-vertex-extension, where $\alpha$ and $\beta$ are arbitrary, and $\gamma$ is the generator of $\Gamma$ which corresponds to 1 in $\mathbb{Z}_k$. In (a,b,c), any one  of the vertices incident to $v$ may be the fixed vertex.}
        \label{extensions image}
\end{figure}

It was shown in \cite{SchulzeLP2024}, that these moves maintain symmetry-generic isostatic properties (see Lemmas 6.5,6.9, 6.12 and 6.14). In the following result, let $|\Gamma|=k$, and let $\gamma$ be the generator of $\Gamma$ which corresponds to $1\in\mathbb{Z}_k$ through an isomorphism.

\begin{lemma}
\label{lemma:rank ext.}
    Let $k\geq4,0\leq j\leq k-1$, and let $(G,\psi,p)$ be a $\rho_j-$symmetrically isostatic $\mathcal{C}_k$-gain framework. Let $(G',\psi')$ be obtained from $(G,\psi)$ by applying an extension. With the same notation as that in Definition~\ref{defn: 1-extension}, assume that if the extension applied is a 1-extension, then the following condition holds:
    \begin{itemize}
        \item [(C1)] $\tau(\psi(e_1))p(v_1),\tau(\psi(e_2))p(v_2)$ and $\tau(\psi(e_3))p(u)$ do not lie on the same line. 
    \end{itemize}
    Assume further that, if the extension applied to $(G,\psi)$ is a loop-1-extension, the following conditions hold:
    \begin{itemize}
        \item[(C2)] If $k$ is even and $j$ is odd, then the new loop does not have gain $\gamma^{k/2}$; and
        \item[(C3)] If the vertex $u$ incident to the new vertex $v$ is fixed, then $j\neq0$ and, whenever $1\leq j\leq k-1$, there is no $n\in S_0(k,j)$ such that $\left<g\right>\simeq\mathbb{Z}_n$, where $g$ is the gain assigned to the new loop.
    \end{itemize}
    Then there is a map $p':V(G')\rightarrow\mathbb{R}^2$ such that $(G',\psi',p')$ is a $\rho_j$-symmetrically isostatic $\mathcal{C}_k$-gain framework.
\end{lemma}

\section{Blockers of a reduction}\label{sec:blockers}

Let $k\geq4,2\leq j\leq k-2$, and let $(G,\psi)$ be a $\mathbb{Z}_k^j$-gain tight $\Gamma$-gain graph. We say a reduction of $(G,\psi)$ is \textit{admissible} if the $\Gamma$-gain graph $(G',\psi')$ which it yields is also $\mathbb{Z}_k^j$-gain tight. It is straightforward to see that  0-reductions, loop-1-reductions and 2-vertex-reductions are always admissible. However, when we apply a 1-reduction to $(G,p)$, we add an edge which may break the sparsity count. If this is the case, we say the graph $(G',\psi')$ obtained from $(G,\psi)$ by applying the 1-reduction has a \textit{blocker}.

\begin{definition}
\label{blocker defn for k-fold rotation}
    Let $k:=|\Gamma|\geq4$, $2\leq j\leq k-2$ and $(G,\psi)$ be a $\mathbb{Z}_k^j$-gain tight $\Gamma$-gain graph. Assume $G$ has a free vertex $v$ of degree 3, with no loop. Let $(G',\psi')$ be a $\Gamma$-gain graph obtained from $(G,\psi)$ by applying a 1-reduction at $v$, and let $e=(v_1,v_2)$ be the edge we add when we apply such reduction. 
    We say a subgraph $H$ of $G-v$ with $v_1,v_2\in V(H)$ and $E(H)\neq\emptyset$ is a \textit{blocker} of $e$ (equivalently, of $(G',\psi')$) if $H+e$ is connected and $|E(H)|=2|V(H)|-3+\alpha_k^j(H+e)$, where $\alpha_k^j$ is as in Definition~\ref{definition of sparsity, higher order rotation}. 
    If $\alpha_k^j(H+e)=3-2|V_0(H)|$, we say $H$ is a \textit{general-count blocker}. If $H+e$ is balanced, we say $H$ is a \textit{balanced blocker}.
\end{definition}

\begin{remark}
    A blocker is defined such that, when joined with the edge added through the 1-reduction, it is connected. However, disconnected graphs may also lead to a break of the sparsity count, when applying a 1-reduction.
    With the same notation as that in Definition~\ref{blocker defn for k-fold rotation}, let $H'$ be a disconnected $\mathbb{Z}_k^j$-gain tight subgraph of $G-v$ with no isolated vertices, such that $v_1,v_2\in V(H')$ and $E(H')\neq\emptyset$. Let $H_1,\dots,H_c$ be the connected components of $H'$. Then, an easy combinatorial argument shows that $H$ is $(2,0,0)$-tight (see, e.g. Lemma 4.13 in \cite{Ikeshita}) and each connected component of $H'$ is also $(2,0,0)$-tight (If, say $|E(H_1)|\leq2|\overline{V(H_1)}|-1$, then some other connected component $H_i$ must satisfy $|E(H_i)|\geq2|\overline{V(H_i)}|+1$, contradicting the sparsity of $(G,\psi)$). For some (not necessarily distinct) $1\leq s,t\leq c$, we have $v_1\in V(H_s),v_2\in V(H_t)$. Then, $H_s\cup H_t$ is a blocker, as given in Definition \ref{blocker defn for k-fold rotation}.
\end{remark}

In this section, we examine the union of two blockers (mostly, we consider blocker whose intersection has non-empty edge set. However Lemma~\ref{lemma:H+v=H+f_1+f_2} allows the intersection of two blockers to have empty edge set). To do so, we need the following results. Lemma~\ref{lemma: there is no tight subgraph containing all three neighbours} was shown in \cite{SchulzeLP2024} (see Proposition 7.2). Here, we give the proof of Lemmas~\ref{there is no fixed cut vertex} and~\ref{lemma:H+v=H+f_1+f_2}.

\begin{proposition}
\label{lemma: there is no tight subgraph containing all three neighbours}
       Let $0\leq m\leq 2,0\leq l\leq3$, let $(G,\psi)$ be a $\Gamma$-gain graph with a free vertex $v$ of degree 3 which has no loop (the neighbours of $v$ need not be distinct). If $(G,\psi)$ is $(2,m,l)$-sparse, then there is no $(2,m,l)$-tight subgraph of $G-v$ which contains all neighbours of $v$.
\end{proposition}

\begin{lemma}
\label{there is no fixed cut vertex}
    Let $0\leq m\leq2,1\leq l\leq3$ be such that $m\leq l$, and let $(G,\psi)$ be a $(2,m,l)$-tight $\Gamma$-gain graph. Then $G$ has no fixed cut-vertex.
\end{lemma}
\begin{proof}
    By Lemma 4.13 in \cite{Ikeshita}, $G$ is connected. Assume, by contradiction, that $v_0\in V_0(G)$ is a cut-vertex of $G$. Let $\{G_1,\dots,G_t\}$ be a partition of $G$ such that $G_i\cap G_j$ is $v_0$ for all $1\leq i\neq j\leq t$, and notice that $E(G_i)\neq\emptyset$ for all $1\leq i\leq t$. It follows that $|E(G_i)|\leq2|\overline{V(G_i)}|+m|V_0(G_i)|-l=2|\overline{V(G_i)}|+m-l$ for all $1\leq i\leq t$. Hence,
    \begin{equation*}
        2|\overline{V(G)}|+m-l=2|\overline{V(G)}|+m|V_0(G)|-l=|E(G)|=\sum_{i=1}^t|E(G_i)|\leq2\sum_{i=1}^t|\overline{V(G_i)}|+mt-lt=2|\overline{V(G)}|+t(m-l).
    \end{equation*}
    It follows that $m-l\leq t(m-l)$. Since $m-l\leq0$, this implies that $t\leq1$. But this contradicts the fact that $v_0$ is a cut-vertex. Hence, the result holds.
\end{proof}

\begin{lemma}
    \label{lemma:H+v=H+f_1+f_2}
    Let $|\Gamma|=k\geq4,2\leq j\leq k-2$, and $(G,\psi)$ be a $\mathbb{Z}_k^j$-gain tight $\Gamma$-gain graph. Assume $G$ has a free vertex $v$ of degree 3, without a loop.  
    Let $(G_1,\psi_1),(G_2,\psi_2)$ be obtained from $(G,\psi)$ by applying two different 1-reductions, which add the edges $f_1,f_2$, respectively. Let $H_1,H_2$ be blockers for $(G_1,\psi_1),(G_2,\psi_2)$, respectively, and use $H$ to denote $H_1\cup H_2$. If $N(v)=3$, assume that $f_1$ and $f_2$ do not share a fixed vertex. Then, $\left<H+v\right>\simeq\left<H+f_1+f_2\right>$.
\end{lemma}

\begin{proof}
    Since $\left<H\right>$ is a subgroup of a cyclic group, we know that there is some integer $n\leq k$ such that $\left<H\right>=\left<h\right>\simeq\mathbb{Z}_n$ through an isomorphism which maps $h$ to 1. We look at the cases where $N(v)$ is 1,2,3, separately. 

    \medskip
    \textit{\textbf{Case 1:} $N(v)=1$.}

    \medskip
    Let $u$ be the neighbour of $v$, let $e_1,e_2,e_2$ be the edges incident to $u$ and $v$, and let $\psi(e_i)=g_i$ for $1\leq i\leq3$. By Lemma~\ref{lemma: forests can have identity gains}, we may assume that $g_1=\textrm{id}$. Moreover, by the definition of gain graph, we know that $g_2,g_3,g_2g_3^{-1}\neq\textrm{id}$. 

    By the definition of 1-reduction and the fact that $(G_1,\psi_1),(G_2,\psi_2)$ are obtained by applying two different 1-reductions, we may assume without loss of generality that $\psi_1(f_1)\neq\psi_2(f_2)$ lie in $\{g_2,g_3,g_2g_3^{-1}\}$. Since $(g_2g_3^{-1})g_3=g_2$ and $(g_2g_3^{-1})^{-1}g_2=g_3$, it follows that $\left<\left\{f_1,f_2\right\}\right>=\left<g_2,g_3\right>$. Similarly, $\left<\left\{e_1,e_2,e_3\right\}\right>=\left<g_2,g_3\right>$. Then, $\left<H+v\right>=\left<H+f_1+f_2\right>=\left<h,g_2,g_3\right>$.

    \medskip
    \textit{\textbf{Case 2:} $N(v)=2$.}

    \medskip
    Let $v_1,v_2$ be the neighbours of $v$, let $e_1,e_1'$ be the edges incident to $v$ and $v_1$, and let $e_2$ be the edge incident to $v$ and $v_2$. By Lemma~\ref{lemma: forests can have identity gains}, we may assume that $\psi(e_1)=\psi(e_2)=\textrm{id}$, and by the definition of gain graph, we know $g:=\psi(e_1')\neq\text{id}$. 

    By the definition of 1-reduction and the fact that $(G_1,\psi_1),(G_2,\psi_2)$ are obtained by applying two different 1-reductions, we know that at most one of $\psi_1(f_1),\psi_2(f_2)$ is $\textrm{id}$, and we may assume without loss of generality that $\psi_i(f_i)\in\{\textrm{id},g\}$ for $1\leq i\leq2$. If $v_2$ is fixed, it follows that $\left<H+v\right>=\left<H+f_1+f_2
    \right>=\left<h,g\right>$. So, assume that $v_2$ is free. 
    
    Let $\mathcal{W}$ be the set of walks from $v_1$ to $v_2$ in $H$ with no fixed vertex and notice that, for all $W\in\mathcal{W}$, $g^{-1}(g\psi(W))=\psi(W)$. Then, $\left<H+v\right>=\left<h,g,\psi(W),g\psi(W):W\in\mathcal{W}\right>=\left<h,g,\psi(W):W\in\mathcal{W}\right>$. Similarly, $\left<H+f_1+f_2\right>=\left<h,g,\psi(W):W\in\mathcal{W}\right>=\left<H+v\right>$.
    
    \medskip
    \textit{\textbf{Case 3:} $N(v)=3$.}

    \medskip
    Let $v_1,v_2,v_3$ be the neighbours of $v$ and, for $1\leq i\leq3$, let $e_i=(v,v_i)$. By Lemma~\ref{lemma: forests can have identity gains}, we may assume that $\psi(e_i)=\textrm{id}$ for $1\leq i\leq3$. Then, by the definition of 1-reduction, $\psi_1(f_1)=\psi_2(f_2)=\textrm{id}$. We may assume, without loss of generality, that $f_1=(v_1,v_2)$ and that $f_2=(v_2,v_3)$. By assumption, $v_2$ is free. For $1\leq s\neq t\leq3$, let $\mathcal{W}_{s,t}$ denote the set of walks from $v_s$ to $v_t$ in $H$ which do not contain a fixed vertex. If $v_1,v_2$ are free, then $\left<H+f_1+f_2\right>,\left<H+v\right>$ are both $\left<h,\psi(W_{12}),\psi(W_{23}),\psi(W_{13}):W_{12}\in\mathcal{W}_{1,2},W_{23}\in\mathcal{W}_{2,3},W_{13}\in\mathcal{W}_{1,3}\right>$. So, we may assume that one of $v_1,v_3$ is fixed. Assume, without loss of generality, that $v_1$ is fixed. Then, $\left<H+f_1+f_2\right>=\left<H+v\right>=\left<h,g,\psi(W):W\in\mathcal{W}_{2,3}\right>$. The result follows.    
\end{proof}

Let $(G,\psi)$ be $\mathbb{Z}_k^j$-gain tight $\Gamma$-gain graph with a free vertex $v$ of degree 3. We aim to show that, expect for a special case which will be dealt with in Section~\ref{sec:final}, there is always an admissible 1-reduction at $v$ (see Theorem~\ref{theorem on 1-reductions}). We will do so using a contradiction argument. It is easy to see that there are at least two possible 1-reductions at $v$ (it can be seen, for instance, in the proof of Lemma~\ref{lemma:H+v=H+f_1+f_2}). Let $(G_1,\psi_1),(G_2,\psi_2)$ be obtained from $(G,\psi)$ by applying two different 1-reductions at $v$, which add the edges $f_1,f_2$, respectively. Suppose that neither one of the 1-reductions is admissible, so that $(G_1,\psi_1),(G_2,\psi_2)$ have some blockers $H_1,H_2$, respectively. 
We start by considering the case where $E(H_1\cap H_2)\neq\emptyset$. For the remaining part of this section, we aim to show that $|E(H_1\cup H_2)|=2|V(H_1\cup H_2)|-3+\alpha_k^j(H_1\cup H_2+f_1+f_2)$ whenever $E(H_1\cap H_2)\neq\emptyset$. 
Then, if $E(H_1\cap H_2)\neq\emptyset$,
we need only consider the case where $H_1\cup H_2+f_1+f_2$ is proper near-balanced and $H_1\cup H_2$ is $(2,1)$-tight, and the case where $N(v)=3$ and $f_1,f_2$ share a fixed vertex (see Corollary~\ref{cor:H+f_1+f_2 bal. or near-bal.}). We will see in Section~\ref{sec:reduct}, that both of these cases also lead to a contradiction. So, we will be able to assume that, given two blockers for two distinct 1-reductions, their intersection has empty edge set.

Since $0\leq\alpha_j^k(H_1+f_1),\alpha_j^k(H_1+f_1)\leq3$, this proof requires to consider  10 different cases and, as a result, it is lengthy. If we restrict the values of $\alpha_j^k(H_1+f_1),\alpha_j^k(H_1+f_1)$ to lie between 1 and 2, we then only have to consider 3 cases. In Section~\ref{sec:alpha<3}, we show that $H_1,H_2$ cannot be general-count blockers, so that $\alpha_j^k(H_1+f_1),\alpha_j^k(H_1+f_1)\leq2$. In Section~\ref{sec:alpha>0}, we show that the desired result holds whenever $\alpha_j^k(H_i+f_i)=0$ for some $1\leq i\leq2$. In Section~\ref{sec:un.block.}, we then prove the full result.

\subsection{The union of two blockers, one of which is a general-count blocker}
\label{sec:alpha<3}
\begin{lemma}
    \label{lemma:no GCB}
    Let $\Gamma$ be a cyclic group of order $k\geq4$. For $2\leq j\leq k-2$, let $(G,\psi)$ be a $\mathbb{Z}_k^j$-gain tight $\Gamma$-gain graph with a free vertex $v$ of degree 3 which has no loop. Let $(G_1,\psi_1),(G_2,\psi_2)$ be obtained from $(G,\psi)$ by applying two different 1-reductions at $v$, which add the edges $f_1$ and $f_2$, respectively. For $i=1,2$, assume that $(G_i,\psi_i)$ has a blocker $H_i$. If $E(H_1\cap H_2)\neq\emptyset$, then $H_1,H_2$ are not general-count blockers.
\end{lemma}

\begin{proof}
    Let $H:=H_1\cup H_2,H':=H_1\cap H_2$, and let $H'_1,\dots, H'_c$ denote the connected components of $H'$. Notice that some of the connected components of $H'$ may be isolated vertices. So, for some integer $0\leq c_0\leq c$, let $H'_1,\dots,H_{c_0}'$ be the isolated vertices of $H'$, and $H_{c_0+1}',\dots,H_c'$ be the connected components of $H'$ with non-empty edge set.
    
    Assume, by contradiction, that $E(H')\neq\emptyset$ and that $H_i$ is a general-count blocker, for some $1\leq i\leq2$. Assume, without loss of generality, that $H_1$ is a general count blocker. We use the abbreviation $\alpha$ to denote $\alpha_k^j(H_2+f_2)$ and, for each $c_0+1\leq i\leq c$, we use $\alpha_i$ to denote $\alpha_k^j(H_i')$. Since $E(H')\neq\emptyset$, we know that $c_0\leq c-1$. By the sparsity of $(G,\psi)$, we have 
    \begin{equation*}
        |E(H')|\leq\sum_{i=1}^{c_0}[2|V(H_i')|-2]+\sum_{i=c_0+1}^c[2|V(H_i')|-3+\alpha_i]=2|V(H')|-(2c_0+3(c-c_0))+\sum_{i=c_0+1}^c\alpha_i.
    \end{equation*}
    Hence,
    \begin{equation*}
        \begin{split}
            |E(H)|&=|E(H_1)|+|E(H_2)|-|E(H')|\\
            &\geq2|\overline{V(H_1)}|+(2|V(H_2)|-3+\alpha)-(2|V(H')|-(2c_0+3(c-c_0))+\sum_{i=c_0+1}^c\alpha_i)\\
            &=2|\overline{V(H_1)}|+(2|\overline{V(H_2)}|+2|V_0(H_2)|-3+\alpha)-(2|\overline{V(H')}|+2|V_0(H')|-(2c_0+3(c-c_0))+\sum_{i=c_0+1}^c\alpha_i)\\
            &=2|\overline{V(H)}|+2(|V_0(H_2)|-|V_0(H')|)+2c_0+3(c-c_0-1)+(\alpha-\sum_{i=c_0+1}^c\alpha_i).
        \end{split}
    \end{equation*}
    Let $f=2(|V_0(H_2)|-|V_0(H')|)+2c_0+3(c-c_0-1)+(\alpha-\sum_{i=c_0+1}^c\alpha_i)$. If we show that $f\geq0$, then $|E(H)|\geq2|\overline{V(H)}|$, and so, by Proposition~\ref{lemma: there is no tight subgraph containing all three neighbours}, the result holds by contradiction. We show that indeed $f\geq0$. 

    To do so, we first note that, for each $c_0+1\leq i\leq c$, $H_i'$ is a subgraph of $H_2+f_2$, and so $\alpha_i\leq\alpha$ whenever $V_0(H_i')=V_0(H_2)$. If $V_0(H')=V_0(H_2)=\emptyset$, it follows that 
    \begin{equation*}
    \begin{split}
        f\geq2c_0+3(c-c_0-1)+(\alpha-(c-c_0)\alpha)=2c_0+(c-c_0-1)(3-\alpha)\geq0,
    \end{split}
    \end{equation*}
    where the last inequality holds because $0\leq c_0\leq c-1$ and $\alpha\leq3$. Hence, we may assume that $V_0(H_2)=\{v_0\}$. By definition, it follows that $\alpha\leq1$. Moreover, since each connected component of $H'$ is a subgraph of $H_2+f_2$, we know that $\alpha_i\leq\alpha+2$ for all $c_0+1\leq i\leq c$. Hence, if $V_0(H')=\emptyset$, it follows that
    \begin{equation*}
    \begin{split}
        f&\geq2+2c_0+3(c-c_0-1)+(\alpha-(c-c_0)(\alpha+2))=(c-c_0-1)(3-\alpha)+2(1-c+2c_0)\\
        &\geq2(c-c_0-1)+2(1-c+2c_0)=2c_0\geq0.
    \end{split}
    \end{equation*}
    So, we may assume that $V_0(H')=\{v_0\}$. If $v_0$ is isolated in $H'$, then $c_0\geq1$. Hence,
    \begin{equation*}
    \begin{split}
        f&\geq2c_0+3(c-c_0-1)+(\alpha-(c-c_0)(\alpha+2))\geq2(c_0-1)\geq0.
    \end{split}
    \end{equation*}
    So assume, without loss of generality, that $v_0\in V(H'_{c_0+1})$. By definition, $\alpha_{c_0+1}\leq\alpha$. Since $\alpha_i\leq\alpha+2$ for all $c_0+2\leq i\leq c$, we have
    \begin{equation*}
    \begin{split}
        f&\geq2c_0+3(c-c_0-1)+(\alpha-\alpha-(\alpha+2)(c-c_0-1))=(c-c_0-1)(1-\alpha)+2c_0\geq0,
    \end{split}
    \end{equation*}
    where the last inequality holds because $0\leq c_0\leq c-1$ and $\alpha\leq1$. We always have $f\geq0$, as required.
\end{proof}

\subsection{The union of two blockers, one of which is (2,3)-tight}
\label{sec:alpha>0}
With the same notation as that in Lemma~\ref{lemma:no GCB}, assume that $\alpha_k^j(H_1+f_1)=0$. By definition, this is equivalent to saying that $H_1+f_1$ is either balanced or $S_0(k,j)$ with $V_0(H_1)=\emptyset$. We consider the two cases separately, in Lemmas~\ref{lemma:bal.int.} and~\ref{lemma:alpha>1}, respectively. However, in Lemma~\ref{lemma:bal.int.}, we do not assume that $H_1+f_1$ is balanced. Instead, we make the slightly weaker assumption that $H_1\cap H_2$ is balanced 
(this weaker assumption will be useful when proving Lemma~\ref{lemma:alpha>1}, as well as Lemma~\ref{lemma: what i need for 1-red.}). 

\begin{lemma}
\label{lemma:bal.int.}
    Let $\Gamma$ be a cyclic group of order $k\geq4$.  For $2\leq j\leq k-2$, let $(G,\psi)$ be a $\mathbb{Z}_k^j$-gain tight $\Gamma$-gain graph with a free vertex $v$ of degree 3 which has no loop. Let $(G_1,\psi_1),(G_2,\psi_2)$ be obtained from $(G,\psi)$ by applying two different 1-reductions at $v$, which add the edges $f_1$ and $f_2$, respectively. For $i=1,2$, assume that $(G_i,\psi_i)$ has a blocker $H_i$, and let $H=H_1\cup H_2$. If $E(H_1\cap H_2)\neq\emptyset$ and $H_1\cap H_2$ is balanced, then $E(H)=2|V(H)|-3+\alpha_k^j(H+f_1+f_2)$.
\end{lemma}

\begin{proof}
Let $H'=H_1\cap H_2$ have connected components $H_1',\dots,H_c'$ and suppose that, for some $c_0\leq c-1$, $H_1',\dots,H_{c_0}$ are isolated vertices and $H_{c_0+1},\dots,H_c$ have non-empty edge set. For $1\leq i\leq2$, use $\alpha_i$ to denote $\alpha_k^j(H_i+f_i)$. We also use $\alpha$ to denote $\alpha_k^j(H+f_1+f_2).$ Assume that $H'$ is balanced. Then, 
\begin{equation*}
        |E(H')|\leq\sum_{i=1}^{c_0}[2|V(H_i')|-2]+\sum_{i=c_0+1}^c[2|V(H_i')|-3]=2|V(H')|-2c_0-3(c-c_0).
    \end{equation*}
Hence, 
\begin{equation}
\label{eq:bal.int.}
    \begin{split}
        |E(H)|&\geq(2|V(H_1)|-3+\alpha_1)+(2|V(H_2)|-3+\alpha_2)-(2|V(H')|-2c_0-3(c-c_0))\\
        &=2|V(H)|-6+\alpha_1+\alpha_2+2c_0+3(c-c_0).
    \end{split}
\end{equation}
If $c-c_0\geq2$, then $|E(H)|\geq2|V(H)|+\alpha_1+\alpha_2+2c_0\geq2|\overline{V(H)}|$, contradicting Proposition~\ref{lemma: there is no tight subgraph containing all three neighbours}. Hence, $c-c_0=1$ and $|E(H)|\geq2|V(H)|-3+\alpha_1+\alpha_2+2c_0$. If $c_0\geq2$, then $|E(H)|\geq2|V(H)|+1$, contradicting the sparsity of $(G,\psi)$. Hence, $(c_0,c_1)$ is either $(0,1)$ or $(1,2)$.

Suppose that $(c_0,c_1)=(1,2)$. By Equation~\eqref{eq:bal.int.}, $|E(H)|\geq2|V(H)|-1+\alpha_1+\alpha_2$. By Proposition~\ref{lemma: there is no tight subgraph containing all three neighbours}, $V_0(H)=\emptyset$ and $\alpha_1=\alpha_2=0$. It follows that $H_1,H_2$ are balanced blockers. By Proposition~\ref{union of [...] and [...] is [...]}(iii), $H+f_1+f_2$ is proper near-balanced, so $\alpha=2$. Then, by the sparsity of $(G,\psi)$, $|E(H)|=2|V(H)|-1=2|V(H)|-3+\alpha$. 

Hence, we may assume that $(c_0,c_1)=(0,1)$ and so, by Equation~\eqref{eq:bal.int.}, 
\begin{equation}
\label{eq:bal.int.2}
    |E(H)|\geq2|V(H)|-3+\alpha_1+\alpha_2.
\end{equation}
By Proposition~\ref{lemma: there is no tight subgraph containing all three neighbours}, $\alpha_1+\alpha_2\leq2$. We look at the cases where $\alpha_1+\alpha_2=2$, $\alpha_1+\alpha_2=1$ and $\alpha_1+\alpha_2=0$ separately. In all such cases, we show that $|E(H)|=2|V(H)|-3+\alpha$, proving the result.

\medskip\textit{\textbf{Case 1: }$\alpha_1+\alpha_2=2$}\\
By Equation~\eqref{eq:bal.int.2}, $|E(H)|\geq2|V(H)|-1$ and so, by the sparsity of $(G,\psi)$, $V_0(H)=\emptyset$. Moreover, $H'$ is $(2,3)$-tight: otherwise, it is easy to see that $|E(H)|\geq2|V(H)|$, contradicting Proposition~\ref{lemma: there is no tight subgraph containing all three neighbours}. Assume, without loss of generality, that $(\alpha_1,\alpha_2)$ is one of $(1,1)$ and $(0,2)$. In the former case, $j$ is odd and $\left<H_1+f_1\right>=\left<H_2+f_2\right>\simeq\mathbb{Z}_2$. Since $H'$ is connected, every closed walk $W$ in $H+f_1+f_2$ can be decomposed as a concatenation of closed walks in $H_1+f_1$ and $H_2+f_2$. It follows, from the fact that $V_0(H)=\emptyset$, that $\left<H+f_1+f_2\right>\simeq\mathbb{Z}_2$. Then, by the sparsity of $(G,\psi)$, $H$ is $(2,1)$-tight, and the result holds. If $(\alpha_1,\alpha_2)=(0,2)$, then $H_1$ is a balanced blocker, and $H_2+f_2$ is either proper near-balanced or $S(k,j)$. In the former case, $H+f_1+f_2$ is proper near-balanced, by Proposition~\ref{union of [...] and [...] is [...]}(ii). In the latter, $H+f_1+f_2$ is $S(k,j)$, by Proposition~\ref{union of balanced and [...] is [...]}(i). In both cases, $\alpha=2$ and $|E(H)|=2|V(H)|-1=2|V(H)|-3+\alpha$, by the sparsity of $(G,\psi)$.

\medskip\textit{\textbf{Case 2: $\alpha_1+\alpha_2=1$}}\\
By Equation~\ref{eq:bal.int.2}, $|E(H)|\geq2|V(H)|-2$. It follows, from Proposition~\ref{lemma: there is no tight subgraph containing all three neighbours}, that $V_0(H)=\emptyset$. Assume, without loss of generality, that $(\alpha_1,\alpha_2)=(1,0).$ Then, $j$ is odd, $\left<H_1+f_1\right>\simeq\mathbb{Z}_2$, and $H_2$ is a balanced blocker. It follows, from Proposition~\ref{union of balanced and [...] is [...]}(i), that $\left<H+f_1+f_2\right>\simeq\mathbb{Z}_2$, and so $\alpha=1$. By the sparsity of $(G,\psi)$, $|E(H)|=2|V(H)|-2=2|V(H)|-3+\alpha$.

\medskip\textit{\textbf{Case 3: $\alpha_1+\alpha_2=0$}}\\
By Equation~\ref{eq:bal.int.2}, $|E(H)|\geq2|V(H)|-3$. Notice that, if $H'$ is not $(2,3)$-tight, then $|E(H)|\geq2|V(H)|-2$ and so $V_0(H)=\emptyset$ by Proposition~\ref{lemma: there is no tight subgraph containing all three neighbours}. It follows, that if $H'$ is not $(2,3)$-tight, then it does not have a fixed cut-vertex. On the other hand, if $H'$ is $(2,3)$-tight, then it does not have a fixed cut-vertex by Lemma~\ref{there is no fixed cut vertex}. So, $H'$ does not have a fixed cut-vertex. For each $1\leq i\leq2$, since $\alpha_i=0$, $H_i$ is either balanced, or it has a fixed vertex and is $S_0(k,j)$. If $H_1,H_2$ are balanced blocker, then $H+f_1+f_2$ is balanced by Proposition~\ref{union of [...] and [...] is [...]}(i). If one of $H_1+f_1,H_2+f_2$ is balanced, and the other is $S_0(k,j)$, then $H+f_1+f_2$ is $S_0(k,j)$ by Proposition~\ref{union of balanced and [...] is [...]}(i), and contains the fixed vertex. If $H_1+f_1,H_2+f_2$ are both $S_0(k,j)$, then so is $H+f_1+f_2$ by Proposition~\ref{gcd proposition}(ii), and it contains the fixed vertex. In all such cases, $\alpha=0$, and $|E(H)|=2|V(H)|-3=2|V(H)|-3+\alpha$, as required.
\end{proof}

\begin{lemma}
\label{lemma:alpha>1}
    Let $\Gamma$ be a cyclic group of order $k\geq4$. For $2\leq j\leq k-2$, let $(G,\psi)$ be a $\mathbb{Z}_k^j$-gain tight $\Gamma$-gain graph with a free vertex $v$ of degree 3 which has no loop. Let $(G_1,\psi_1),(G_2,\psi_2)$ be obtained from $(G,\psi)$ by applying two different 1-reductions at $v$, which add the edges $f_1$ and $f_2$, respectively. For $i=1,2$, assume that $(G_i,\psi_i)$ has a blocker $H_i$. Assume further that $V_0(H_1)=\{v_0\}$ and that $H_1+f_1$ is $S_0(k,j)$. If $E(H_1\cap H_2)\neq\emptyset$, then $H:=H_1\cup H_2$ satisfies $|E(H)|=2|V(H)|-3+\alpha_k^j(H+f_1+f_2)$.
\end{lemma} 

\begin{proof}
    Let $H':=H_1\cap H_2$, let $H_1,\dots,H_{c_0}$ be the isolated vertices of $H'$, and $H_{c_0+1},\dots,H_c$ be the connected components of $H'$ with non-empty edge set. Assume that $c\geq c_0+1$. By Lemma~\ref{lemma:bal.int.}, we may assume that $H'$ is unbalanced. In particular, $H_2$ is not a balanced blocker. Moreover, by Lemma~\ref{lemma:no GCB}, we may assume that $H_2$ is not a general-count blocker. Throughout the proof, let $\alpha$ denote $\alpha_k^j(H_2+f_2)$. We look at the cases where $V_0(H')=\emptyset$ and $V_0(H')=\{v_0\}$ separately.

    First, suppose that $V_0(H')=\emptyset$. Since $V_0(H_1)=\{v_0\},$ it follows that $V_0(H_2)=\emptyset$. By assumption, this implies that $1\leq\alpha\leq2$. Since each connected component of $H'$ is a subgraph of $H_2+f_2$ and $V_0(H')=V_0(H_2)=\emptyset$,
    \begin{equation*}
        |E(H')|=\sum_{i=1}^c|E(H_i')|\leq\sum_{i=1}^{c_0}[2|V(H_i')|-2]+\sum_{i=1+c_0}^c[2|V(H_i')|-3+\alpha]=2|V(H')|-2c_0+(c-c_0)(\alpha-3).
    \end{equation*}
    Hence, 
    \begin{equation}
    \label{eq:S_0(k,j) blocker}
    \begin{split}
        |E(H)|&=|E(H_1)|+|E(H_2)|-|E(H')|\\
        &\geq(2|V(H_1)|-3)+(2|V(H_2)|-3+\alpha)-(2|V(H')|-2c_0+(c-c_0)(\alpha-3))\\
        &=2|V(H)|-6+\alpha+2c_0+(c-c_0)(3-\alpha)=2|\overline{V(H)}|-4+\alpha+2c_0+(c-c_0)(3-\alpha).
    \end{split}
    \end{equation}
    We show that $c_0=0$ and $c_1=1$. Assume, by contradiction, that $c-c_0\geq2$. Then, by Equation~\eqref{eq:S_0(k,j) blocker} and the fact that $\alpha\leq2$, we have $|E(H)|\geq2|\overline{V(H)}|+2-\alpha\geq2|\overline{V(H)}|$. This contradicts Proposition~\ref{lemma: there is no tight subgraph containing all three neighbours}. Hence, $c=c_0+1$ and $|E(H)|\geq2|\overline{V(H)}|-1+2c_0$, by Equation~\ref{eq:S_0(k,j) blocker}. By Proposition~\ref{lemma: there is no tight subgraph containing all three neighbours}, it follows that $c_0=0,c=1$ and $|E(H)|=2|\overline{V(H)}|-1$. If we show that $H+f_1+f_2$ is $S_0(k,j)$, then $|E(H)|=2|V(H)|-3+\alpha_k^j(H+f_1+f_2)$, as required. We show that $H+f_1+f_2$ is indeed $S_0(k,j)$. Since $1\leq\alpha\leq2$ and $V_0(H_2)=\emptyset$, exactly one of the following holds: $j$ is odd and $\left<H_2+f_2\right>\simeq\mathbb{Z}_2$; $H_2+f_2$ is $S(k,j)$; $H_2+f_2$ is proper near-balanced.
    If $\left<H_2+f_2\right>\simeq\mathbb{Z}_2$, then $H+f_1+f_2$ is $S_0(k,j)$ by Proposition~\ref{union of balanced and [...] is [...]}(ii). If $H_2+f_2$ is $S(k,j)$, then it is $S_0(k,j)$ by Proposition~\ref{union of [...] and [...] is [...]}(i). Hence, $H+f_1+f_2$ is $S_0(k,j)$ by Proposition~\ref{union of [...] and [...] is [...]}(ii). If $H_2+f_2$ is near-balanced, then $H+f_1+f_2$ is $S_0(k,j)$ by Proposition~\ref{union of [...] and [...] is [...]}(iv). So, whenever $V_0(H')=\emptyset$, the result holds.

    Now, assume that $V_0(H')=\{v_0\}$. This implies that $V_0(H_2)\neq\emptyset$. Hence, $|E(H_2)|=2|\overline{V(H_2)}|-1+\alpha$. If $v_0$ is isolated in $H'$, then $c_0\geq1$. Assume, without loss of generality, that $v_0$ is $H_1'$. Since each $H_i'$ is a subgraph of $H_1+f_1$, we have
    \begin{equation*}
    |E(H')|=\sum_{i=1}^c|E(H_i')|\leq2|\overline{V(H_1')}|+\sum_{i=2}^{c_0}[2|\overline{V(H_i')}|-2]+\sum_{i=c_0+1}^c[2|\overline{V(H_i')}|-1]=2|\overline{V(H')}|-2(c_0-1)-(c-c_0),
    \end{equation*}
    and so
    \begin{equation*}
        \begin{split}
            |E(H)|&=|E(H_1)|+|E(H_2)|-|E(H')|\\
            &\geq(2|\overline{V(H_1)}|-1)+(2|\overline{V(H_2)}|-1+\alpha)-(2|\overline{V(H')}|-2(c_0-1)-(c-c_0))\\
            &=2|\overline{V(H)}|-2+\alpha+2(c_0-1)+(c-c_0).
        \end{split}
    \end{equation*}
    If $c-c_0\geq2$ or if $c_0\geq2$, this contradicts Proposition~\ref{lemma: there is no tight subgraph containing all three neighbours}. Hence, we may assume that $c_0=1,c=2$. So, $|E(H)|\geq2|\overline{V(H)}|-1+\alpha$. 

    In a similar way, if $v_0$ is not an isolated vertex of $H'$, we can see that $|E(H)|\geq2|\overline{V(H)}|-2+\alpha+2c_0+(c-c_0).$ If $c_0\geq1$ or $c-c_0\geq2$, this contradicts Proposition~\ref{lemma: there is no tight subgraph containing all three neighbours}. Hence, $c_0=0,c=1$, and $|E(H)|\geq2|\overline{V(H)}|-1+\alpha$. Both when $v_0$ is an isolated vertex of $H'$ and when it isn't, Proposition~\ref{lemma: there is no tight subgraph containing all three neighbours} implies that $\alpha=0$ and $|E(H)|=2|\overline{V(H)}|-1$. Hence, it is enough show that $H+f_1+f_2$ is $S_0(k,j)$. Since $\alpha=0$ and $H_2$ is not a balanced blocker, $H_2+f_2$ is $S_0(k,j)$. Moreover, $H'$ is either connected, or it is composed of two connected components, one of which is the isolated fixed vertex. So, $H+f_1+f_2$ is $S_0(k,j)$ by Proposition~\ref{gcd proposition}(ii), and the result holds.
\end{proof}

\subsection{The union of two blockers with non-empty edge set}
\label{sec:un.block.}
\begin{lemma}
\label{lemma: what i need for 1-red.}
    Let $\Gamma$ be a cyclic group of order $k\geq4$. For $2\leq j\leq k-2$, let $(G,\psi)$ be a $\mathbb{Z}_k^j$-gain tight $\Gamma$-gain graph with a free vertex $v$ of degree 3 which has no loop. Let $(G_1,\psi_1),(G_2,\psi_2)$ be obtained from $(G,\psi)$ by applying two different 1-reductions at $v$, which add the edges $f_1$ and $f_2$, respectively. For $i=1,2$, assume that $(G_i,\psi_i)$ has a blocker $H_i$. If $E(H_1\cap H_2)\neq\emptyset$, then $|E(H)|=2|V(H)|-3+\alpha_k^j(H+f_1+f_2)$. 
\end{lemma}

\begin{proof}   
    Let $H'=H_1\cap H_2$ have connected components $H_1',\dots,H_c'$ and suppose that, for some $c_0\leq c-1$, $H_1',\dots,H_{c_0}$ are isolated vertices and $H_{c_0+1},\dots,H_c$ have non-empty edge set. We abbreviate $\alpha_k^j(H_i+f_i)$ to $\alpha_i$, for $i=1,2$. By Lemma~\ref{lemma:bal.int.}, we may assume that $H'$ is unbalanced. Moreover, by Lemmas~\ref{lemma:no GCB},~\ref{lemma:bal.int.} and~\ref{lemma:alpha>1}, we may assume that $1\leq\alpha_1,\alpha_2\leq2$. Without loss of generality, assume that $\alpha_1\geq\alpha_2$. We look at the cases where $(\alpha_1,\alpha_2)=(1,1),(2,1),(2,2)$, separately. 
    
    \medskip
    \textit{\textbf{Case 1: }$\alpha_1=\alpha_2=1$.} 

    If we show that $V_0(H')=\emptyset$ then, by the definition of $\alpha_1,\alpha_2$, $j$ is odd and $\left<H_i+f_i\right>\simeq\mathbb{Z}_2$ for some $1\leq i\leq2$. We show that $V_0(H')$ is indeed empty. So assume, by contradiction, that $|V_0(H')|=1$. By the sparsity of $(G,\psi)$, we have $|E(H_i')|\leq 2|\overline{V(H_i')}|$ for $i=c_0+1,\ldots, c$. If the fixed vertex is isolated, then $c_0\geq1$ and so
    \begin{equation*}
        |E(H')|=\sum_{i=1}^c|E(H_i')|\leq\sum_{i=1}^{c_0}[2|V(H_i')|-2]+\sum_{i=c_0+1}^c2|V(H_i')|=2|V(H')|-2c_0\leq2|V(H')|-2.
    \end{equation*}
    If the fixed vertex is not isolated, assume without loss of generality, that it lies in $H_{c_0+1}'$. Then, 
    \begin{equation*}
        |E(H')|=\sum_{i=1}^c|E(H_i')|\leq\sum_{i=1}^{c_0}[2|V(H_i')|-2]+[2|V(H_{1+c_0}')|-2]+\sum_{i=c_0+2}^c2|V(H_i')|=2|V(H')|-2c_0-2.
    \end{equation*}
    Since $c_0\geq0$, $|E(H')|\leq2|V(H')|-2$. Hence, in both cases we have 
    \begin{equation*}
        \begin{split}
            |E(H)|\geq(2|V(H_1)|-2)+(2|V(H_2)|-2)-(2|V(H')|-2)=2|V(H)|-2=2|\overline{V(H)}|.
        \end{split}
    \end{equation*}
    By the sparsity of $(G,\psi)$ and Proposition \ref{lemma: there is no tight subgraph containing all three neighbours}, this is a contradiction. So, $V_0(H')=\emptyset$, $j$ is odd and $\left<H_i+f_i\right>\simeq\mathbb{Z}_2$ for some $1\leq i\leq2$. Assume, without loss of generality, that $\left<H_1+f_1\right>\simeq\mathbb{Z}_2$. Then, since $H'$ is a subgraph of $H_1+f_1$ and $j$ is odd, $|E(H')|\leq2|V(H')|-2c$, and so
    \begin{equation}
    \label{eq: case 2b}
        \begin{split}
            |E(H)|\geq(2|V(H_1)|-2)+(2|V(H_2)|-2)-(2|V(H')|-2c)=2|V(H)|+2(c-2).
        \end{split}
    \end{equation}
    By the sparsity of $(G,\psi)$ and Proposition \ref{lemma: there is no tight subgraph containing all three neighbours}, this implies that $c=1$ and that $|V_0(H)|=0$. Hence, $|V_0(H_2)|=0$, and we have $\left<H_2+f_2\right>\simeq\mathbb{Z}_2$. Since $H'$ is connected, every closed walk $W$ in $H+f_1+f_2$ can be decomposed as a concatenation of closed walks in $H_1+f_1$ and $H_2+f_2$. Hence, $\left<H+f_1+f_2\right>\simeq\mathbb{Z}_2$. By the sparsity of $(G,\psi)$, and by Equation~\eqref{eq: case 2b}, $|E(H)|=2|V(H)|-3+\alpha_k^j(H+f_1+f_2)$.

    \medskip
    \textit{\textbf{Case 2:} $\alpha_1=2,\alpha_2=1$.}
    
    By the definition of $\alpha_1$, $|V_0(H_1)|=0$ and $H_1+f_1$ is either $S(k,j)$ or near-balanced. Notice that for each $1\leq i\leq c_0$, $|E(H'_i)|=2|V(H'_i)|-2<2|V(H'_i)|-1$. So, since $|V_0(H')|=0$ and $H'$ is a subgraph of $H_1+f_1$, $|E(H')|\leq\sum_{i=1}^c[2|V(H_i')|-1]=2|V(H')|-c$. Hence,
    \begin{equation}
    \label{eq. for alpha=2}
        \begin{split}
            |E(H)|\geq(2|V(H_1)|-1)+(2|V(H_2)|-2)-(2|V(H')|-c)=2|V(H)|-3+c\geq2|V(H)|-2,
        \end{split}
    \end{equation}
    since $c\geq1$. By Proposition \ref{lemma: there is no tight subgraph containing all three neighbours}, $|V_0(H)|=0$. By the definition of $\alpha_2$, this implies that $j$ is odd and $\left<H_2+f_2\right>\simeq\mathbb{Z}_2$. Then, since $H'$ is a subgraph of $H_2+f_2$, each connected component of $H'$ must be $(2,2)$-sparse. It follows that 
    $|E(H')|\leq2|V(H')|-2c$ and
    \begin{equation*}
        |E(H)|\geq(2|V(H_1)|-1)+(2|V(H_2)|-2)-(2|V(H')|-2c)=2|V(H)|+2c-3.
    \end{equation*}
    This implies that $c=1$, by the sparsity of $(G,\psi)$. Since $H'$ is unbalanced, $H_1+f_1$ is not proper near-balanced: otherwise, $\left<H_1+f_1\right>\simeq\mathbb{Z}_2$, by Proposition~\ref{union of [...] and [...] is [...]}(iii), which contradicts the definition of proper near-balancedness. It follows that $H_1+f_1$ is $S(k,j)$. Then, by Proposition~\ref{union of balanced and [...] is [...]}(ii), $H+f_1+f_2$ is $S(k,j)$ and so $\alpha_k^j(H+f_1+f_2)=2$. Hence, 
    \begin{equation}
    \label{eq: case 3b}
        |E(H)|\geq2|V(H)|-1= 2|V(H)|-3+\alpha_k^j(H+f_1+f_2).
    \end{equation}
    By the sparsity of $(G,\psi)$, Equation~\eqref{eq: case 3b} holds with equality.

    \medskip\textit{\textbf{Case 3: }$\alpha_1=\alpha_2=2.$}

    In a similar way as we did in Case 2, we can see that $|E(H)|\geq2|V(H)|-2+c$. If $c\geq2$ or if $V_0(H)\neq\emptyset$, then $|E(H)|\geq2|V(H)|$, contradicting Proposition~\ref{lemma: there is no tight subgraph containing all three neighbours}. So $c=1$ and $V_0(H)=\emptyset$. Since $H'$ is a subgraph of $H_1+f_1$ and $V_0(H')=V_0(H_1+f_1)=\emptyset$, it is $(2,1)$-sparse. If $|E(H')|\leq2|V(H')|-2$, it is easy to see that $|E(H)\geq2|V(H)|$, contradicting Proposition~\ref{lemma: there is no tight subgraph containing all three neighbours}. Hence, $H'$ is $(2,1)$-tight.
    
    If exactly one of $H_1+f_1,H_2+f_2$ is near-balanced, then $H+f_1+f_2$ is $S(k,j)$ by Proposition~\ref{gcd proposition}(iii).
    If both $H_1+f_1,H_2+f_2$ are $S(k,j)$, then they are both $S_i(k,j)$ for some $i\in\{0,-1,1\}$, by Proposition~\ref{gcd proposition}(i). So, by Proposition~\ref{gcd proposition}(ii), $H+f_1+f_2$ is also $S_i(k,j)$. If neither $H_1+f_1$ nor $H_2+f_2$ is $S(k,j)$, then they are both proper near-balanced. Hence, $H'$ is also proper near-balanced and so $H+f_1+f_2$ is near-balanced by Proposition~\ref{union of [...] and [...] is [...]}(i). By the sparsity of $(G,\psi)$ and Proposition \ref{lemma: there is no tight subgraph containing all three neighbours}, $|E(H)|=2|V(H)|-1$ and $|V_0(H)|=0$. Since $\alpha_k^j(H+f_1+f_2)=2$, we have $|E(H)|=2|V(H)|-3+\alpha_k^j(H+f_1+f_2)$, as required.  
\end{proof}

Proposition~\ref{lemma: there is no tight subgraph containing all three neighbours}, and Lemmas~\ref{lemma:H+v=H+f_1+f_2},~\ref{lemma: what i need for 1-red.} imply the following result.

\begin{corollary}
\label{cor:H+f_1+f_2 bal. or near-bal.}
    Let $\Gamma$ be a cyclic group of order $k\geq4$. For $2\leq j\leq k-2$, let $(G,\psi)$ be a $\mathbb{Z}_k^j$-gain tight $\Gamma$-gain graph with a free vertex $v$ of degree 3 which has no loop. Let $(G_1,\psi_1),(G_2,\psi_2)$ be obtained from $(G,\psi)$ by applying two different 1-reductions at $v$, which add the edges $f_1$ and $f_2$, respectively. For $i=1,2$, assume that $(G_i,\psi_i)$ has a blocker $H_i$, and suppose that $E(H_1\cap H_2)\neq\emptyset$. If $N(v)\neq3$, or if $f_1$ and $f_2$ do not share a fixed vertex, then $H_1\cup H_2+f_1+f_2$ is proper near-balanced.
\end{corollary}
\begin{proof}
   Let $H=H_1\cup H_2$. Assume that $N(v)\neq3$, or that $f_1$ and $f_2$ do not share a fixed vertex. Assume, by contradiction, that $H+f_1+f_2$ is not proper near-balanced. By Lemma~\ref{lemma:H+v=H+f_1+f_2}, $\left<H+f_1+f_2\right>=\left<H+v\right>$. Then we have $\alpha_j(H+v)=\alpha_j(H+f_1+f_2)$. By Lemma~\ref{lemma: what i need for 1-red.}, $|E(H)|=2|V(H)|-3+\alpha_k^j(H+v)$, which contradicts Proposition~\ref{lemma: there is no tight subgraph containing all three neighbours}. Hence, $H+f_1+f_2$ is proper near-balanced.
\end{proof}

\section{A gain-tight graph admits a reduction} \label{sec:reduct}

The following result is crucial for the combinatorial results of the paper. We show that, given a vertex $v$ of degree 3, we may always apply an admissible 1-reduction at $v$  except in one special case.

\begin{theorem}
\label{theorem on 1-reductions}
For $k\geq4,$ let $\Gamma=\left<\gamma\right>\simeq\mathbb{Z}_k$ through the isomorphism defined by letting $\gamma\mapsto1$.
Let $(G,\psi)$ be a $\Gamma$-gain graph with a free vertex $v$ of degree 3 and no loop. Suppose that $(G,\psi)$ is $\mathbb{Z}_k^j$-tight for some $2\leq j\leq k-2$. If there is not an admissible $1$-reduction at $v$, then  
    $k$ is even and $j$ is odd, $v$ has exactly two neighbours, only one of which is free, call it $v_1$. Moreover, the $2$-cycle $v,v_1,v$ has gain $\gamma^{k/2}$ (see Figure~\ref{degree 2 vertex problem image pt. 2}). 

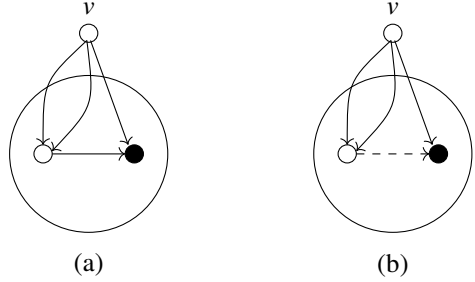
\begin{figure}[H]
    \centering
    \begin{tikzpicture}[scale=.8]
        \draw(3.5,0) circle (0.15cm);  
  \draw (4.25,0) circle (1.3cm);
  \draw[fill=black] (5,0) circle (0.15cm);
  \draw (4.25, 2) circle (0.15cm);
  \node[above] at (4.25, 2.15){$v$};
  \draw[->] (4.28, 1.88) -- (4.9, 0.15);
  \draw[->] (4.22,1.88) .. controls (4.35,0.8) .. (3.65, 0.05);
  \draw[->] (4.22,1.88) .. controls (3.5,1.2) .. (3.5, 0.15);
  \draw[->] (3.65,0) -- (4.85,0);
  \node[below] at (4.25, -1.5) {(a)};

  \draw(8.5,0) circle (0.15cm);  
  \draw (9.25,0) circle (1.3cm);
  \draw[fill=black] (10,0) circle (0.15cm);
  \draw (9.25, 2) circle (0.15cm);
  \node[above] at (9.25, 2.15){$v$};
  \draw[->] (9.28, 1.88) -- (9.9, 0.15);
  \draw[->] (9.22,1.88) .. controls (9.35,0.8) .. (8.65, 0.05);
  \draw[->] (9.22,1.88) .. controls (8.5,1.2) .. (8.5, 0.15);
  \draw[->, dashed] (8.65,0) -- (9.85,0);
  \node[below] at (9.25, -1.5) {(b)};

    \end{tikzpicture}
    \caption{Two instances of a vertex $v$ of degree 3. In both cases $v$ has two neighbours, one of which is fixed. In (a) there is an edge between the neighbours of $v$, in (b) there isn't.}
    \label{degree 2 vertex problem image pt. 2}
\end{figure}
\end{theorem}

The proof of Theorem~\ref{theorem on 1-reductions} is lengthy. Hence, we look at the cases  $N(v)=1,2,3$, separately.

\subsection{Applying a 1-reduction at a vertex with exactly one neighbour}
\begin{proposition}
    \label{thm:N(v)=1}
    For $k\geq4$, let $\Gamma=\left<\gamma\right>\simeq\mathbb{Z}_k$ through the isomorphism defined by $\gamma\mapsto1$. For $2\leq j\leq k-2$, let $(G,\psi)$ be a $\mathbb{Z}_k^j$-gain tight $\Gamma$-gain graph with a vertex $v$ of degree 3. Suppose that $v$ has no loop, and exactly one neighbour $u$. Then, there is an admissible 1-reduction at $v$.
\end{proposition}
\begin{proof}
    Notice that $u,v$ must be free, since they have parallel edges. Let $e_1,e_2,e_3$ be the edges incident to $u$ and $v$, with $g_i:=\psi(e_i)$ for $1\leq i\leq3$. By Lemma~\ref{lemma: forests can have identity gains}, we may assume that $g_1=\textrm{id}$. Moreover, $g_2,g_3,g_2g_3^{-1}\neq\textrm{id}$ by the definition of gain graph. Let $(G_1,\psi_1),(G_2,\psi_2)$ and $(G_3,\psi_3)$ be obtained from $G-v$ by adding the loops $f_1,f_2,f_3$ at $u$ with gains $g_2,g_3,g_2g_3^{-1}$, respectively. Assume, by contradiction, that for each $1\leq i\leq3$, $(G_i,\psi_i)$ has a blocker $H_i$, and for all such $i$ let $\alpha_i$ denote $\alpha_k^j(H_i+f_i)$. Notice that, for each $1\leq i\leq3$, $H_i$ is neither a balanced blocker (since $H_i+f_i$ contains a loop), nor a general-count blocker (by Proposition~\ref{lemma: there is no tight subgraph containing all three neighbours}). Since $g_2,g_3,g_2g_3^{-1}\neq\textrm{id}$, at most one of $g_2,g_3,g_2g_3^{-1}$ is $\gamma^{k/2}$, and so at most one of $\left<H_1+f_1\right>,\left<H_2+f_2\right>,\left<H_3+f_3\right>$ is isomorphic to $\mathbb{Z}_2$.

    Notice that, for all $1\leq s\neq t\leq3$, $H_s\cup H_t+f_s+f_t$ contains a vertex with two different loops, and so it is not proper near-balanced. It follows, from Corollary~\ref{cor:H+f_1+f_2 bal. or near-bal.} that $E(H_s\cap H_t)=\emptyset$ for all $1\leq s\neq t\leq3$.
    
    We now show that at most one of $H_1+f_1,H_2+f_2,H_3+f_3$ is $S(k,j)$. To do so, fix some $1\leq s\neq t\leq3$ and assume, by contradiction, that $H_s+f_s,H_t+f_t$ are both $S(k,j)$. Then, $H_s$ is $(2,m_s,1)$-tight and $H_t$ is $(2,m_t,1)$-tight, for some $0\leq m_s,m_t\leq1$. Since $u\in V(H_s\cap H_t)$ is free, we have
    \begin{equation*}
    \begin{split}
        |E(H_s\cup H_t)|&=|E(H_s)|+|E(H_t)|\\
        &=(2|\overline{V(H_s)}|+m_s|V_0(H_s)|-1)+(2|\overline{V(H_t)}|+m_t|V_0(H_t)|-1)\\
        &=2|\overline{V(H_s\cup H_t)}|+2|\overline{V(H_s\cap H_t)}|-2+m_s|V_0(H_s)|+m_t|V_0(H_t)|\\
        &\geq2|\overline{V(H_s\cup H_t)}|,
    \end{split}
    \end{equation*}
    contradicting Proposition~\ref{lemma: there is no tight subgraph containing all three neighbours}. So, we may assume that at most one of $H_1+f_1,H_2+f_2,H_3+f_3$ is $S(k,j)$.  
    This implies that, for some $1\leq i\leq3$, $H_i+f_i$ is proper near-balanced (since none of the $H_i$ is a balanced blocker or a general-count blocker,  at most one of the $H_i+f_i$ is $S(k,j)$ and for at most one of the $H_i+f_i$ we have $\left<H_i+f_i\right>=\mathbb{Z}_2$.), and so $\alpha_i=2$. Without loss of generality, assume that $\alpha_3=2$. Let $H:=H_1\,\cup\,H_2\,\cup\,H_3$, and $H':=H_1\,\cap\,H_2\,\cap\,H_3$. Since $u\in V(H_s\cap H_t)$ for all $1\leq s\neq t\leq3$, we have
    \begin{equation*}
        \begin{split}
            |E(H)|&=\sum_{i=1}^3|E(H_i)|=2\sum_{i=1}^3|V(H_i)|-9+\sum_{i=1}^3\alpha_i\\
            &=2|V(H)|+2\sum_{1\leq s\neq t\leq3}[|V(H_s\cap H_t)|-|V(H')|]-7+\alpha_1+\alpha_2\\
            &\geq2|V(H)|-3+\alpha_1+\alpha_2.
        \end{split}
    \end{equation*}
    So, $\alpha_1+\alpha_2\leq2$, by Proposition~\ref{lemma: there is no tight subgraph containing all three neighbours} and the sparsity of $(G,\psi)$. If $H$ has a fixed vertex, then we have $|E(H)|\geq2|V(H)|-3+\alpha_1+\alpha_2=2|\overline{V(H)}|-1+\alpha_1+\alpha_2,$ and so $\alpha_1+\alpha_2=0$, by Proposition~\ref{lemma: there is no tight subgraph containing all three neighbours} and the sparsity of $(G,\psi)$. Since $H_1,H_2$ are not balanced blockers, the only case in which $\alpha_1=\alpha_2=0$ is when $H_1+f_1,H_2+f_2$ are both $S_0(k,j)$ and $|V_0(H_1)|=|V_0(H_2)|=1$. But this contradicts the fact that at most one of $H_1+f_1,H_2+f_2,H_3+f_3$ is $S(k,j)$. So, we may assume that $V_0(H)=\emptyset$. This implies that, for $i=1,2$, $\alpha_i\geq1$ with equality if and only if $\left<H_i+f_i\right>\simeq\mathbb{Z}_2$ and $j$ is odd. So, the only way of having $\alpha_1+\alpha_2\leq2$ is if $\alpha_1=\alpha_2=1$ and $\left<H_1+f_1\right>=\left<H_2+f_2\right>\simeq\mathbb{Z}_2$, with odd $j$. This contradicts the fact that at most one of $\left<H_1+f_1\right>,\left<H_2+f_2\right>,\left<H_3+f_3\right>$ is isomorphic to $\mathbb{Z}_2$. By contradiction, the result holds.
\end{proof}

\subsection{Applying a 1-reduction at a vertex with exactly two distinct neighbours}
\begin{proposition}
\label{pro: N(v)=2,v2 free}
For $k\geq4$, let $\Gamma=\left<\gamma\right>\simeq\mathbb{Z}_k$ through the isomorphism defined by letting $\gamma\mapsto1$.
For $2\leq j\leq k-2$, let $(G,\psi)$ be a $\mathbb{Z}_k^j$-gain tight $\Gamma$-gain graph with a free vertex $v$ of degree 3. Suppose that $v$ has no loop, and exactly two distinct neighbours $v_1,v_2$. Suppose that $v_1,v_2$ are free. Then there is an admissible 1-reduction at $v$.
\end{proposition}

\begin{proof}
    Let $e_1,e'_1$ be the edges from $v$ to $v_1$, and $e_2$ be the edge from $v$ to $v_2$.
    By Lemma~\ref{lemma: forests can have identity gains}, we may assume that $\psi(e_1)=\psi(e_2)=\textrm{id}$ and, by the definition of gain graph, we know that $g:=\psi(e_1')\neq \textrm{id}$. 

    Let $(G_1,\psi_1),(G_2,\psi_2),(G_3,\psi_3)$ be obtained from $G-v$ by adding, respectively, the edges $f_1=(v_1,v_2)$ with gain $\textrm{id}$, the edge $f_2=(v_2,v_1)$ with gain $g$, and a loop $f_3$ at $v_1$ with gain $g$. Assume, by contradiction, that $H_1,H_2$ and $H_3$ are blockers for $(G_1,\psi_1),(G_2,\psi_2)$ and $(G_3,\psi_3)$, respectively. Let $H=H_1\cup H_2\cup H_3$ and $H'=H_1\cap H_2\cap H_3$. By Proposition~\ref{lemma: there is no tight subgraph containing all three neighbours}, $H_1,H_2$ are not general-count blockers. Moreover, $H_3$ is not a balanced blocker, since $H_3+f_3$ contains a loop. 
    
    We start by showing that $E(H_s\cap H_t)=\emptyset$ for all $1\leq s\neq t\leq3$. So, assume by contradiction that $E(H_s\cap H_t)\neq\emptyset$ for some $1\leq s\neq t\leq3$. By Corollary~\ref{cor:H+f_1+f_2 bal. or near-bal.}, $H_s\cup H_t+f_s+f_t$ is proper near-balanced. Moreover, by Lemma~\ref{lemma: what i need for 1-red.}, $H_s\cup H_t$ is $(2,1)$-tight. 
    
    In particular, if $s=1,t=2$, then the base-vertices of near-balancedness must be $v_1,v_2$: otherwise, there is a gain $\psi'$ equivalent to $\psi$ such that $\psi'(f_1)=\psi'(f_2)=\textrm{id}$, contradicting the definition of gain graph. This implies that every path $W$ from $v_1$ to $v_2$ in $H_1\cup H_2$ has gain $\textrm{id}$, or $g^{-1}$: $W$ must have gain in $\{\textrm{id},g,g^{-1}\}$ because $f_1\in E(H_1\cup H_2+f_1+f_2)$, and it cannot have gain $g$, because $f_2\in E(H_1\cup H_2+f_1+f_2)$. Then $H_1\cup H_2+v$ is also near-balanced. Since $H$ is $(2,1)$-tight, this contradicts Proposition~\ref{lemma: there is no tight subgraph containing all three neighbours}. Hence, $E(H_1\cap H_2)=\emptyset$, and one of $s,t$ is 3.

    Assume, without loss of generality, that $E(H_1\cap H_3)\neq\emptyset$, and recall that this implies that $H_1\cup H_3+f_1+f_3$ is proper near-balanced, and that $H_1\cup H_3$ is $(2,1)$-tight. By the sparsity of $(G,\psi)$, $H_1\cup H_3+f_1$ is also proper near-balanced. It follows that $H_1':=H_1\cup H_3$ is a blocker for $(G_1,\psi_1)$. If $E(H_2\cap H_3)\neq\emptyset$, then the same argument shows that $H_2':=H_2\cup H_3$ is a blocker for $(G_2,\psi_2)$. Since $E(H_1'\cap H_2')=E(H_3)\neq\emptyset$, $H_1'\cup H_2'+f_1+f_2$ is proper-near balanced, by Corollary~\ref{cor:H+f_1+f_2 bal. or near-bal.}, and $H'_1\cup H_2'$ is $(2,1)$-tight by Lemma~\ref{lemma: what i need for 1-red.}. Using a similar argument as in the previous paragraph, we can see that $H_1'\cup H_2'+v$ is proper near-balanced, contradicting Proposition~\ref{lemma: there is no tight subgraph containing all three neighbours}. Hence, $E(H_2\cap H_3)=\emptyset$. It follows that
    \begin{equation*}
            \begin{split}
                |E(H)|&=|E((H_1\cup H_3)\cup H_2)|=|E(H_1\cup H_3)|+|E(H_2)|=(2|V(H_1\cup H_3)|-1)+(2|V(H_2)|-3+\alpha_2)\\
                &=2|V(H)|+2|V((H_1\cup H_3)\cap H_2)|-4+\alpha_2\geq2|V(H)|+\alpha_2\geq2|V(H)|\geq2|\overline{V(H)}|,
            \end{split}
        \end{equation*}
        since $v_1,v_2\in V(H_1),V(H_2)$ and $\alpha_2\geq0$. This contradicts Proposition~\ref{lemma: there is no tight subgraph containing all three neighbours}. Hence, $E(H_s\cap H_t)=\emptyset$ for all $1\leq s\neq t\leq 3$. 

        Since $E(H_1\cap H_2)=\emptyset$, 
        \begin{equation*}
        \begin{split}
            |E(H_1\cup H_2)|&=|E(H_1)|+|E(H_2)|=(2|V(H_1)|-3+\alpha_1)+(2|V(H_2)|-3+\alpha_2)\\
            &=2|V(H_1\cup H_2)|+2|V(H_1\cap H_2)|-6+\alpha_1+\alpha_2.
        \end{split}
        \end{equation*}
        If $|V(H_1\cap H_2)|\geq3$, or if $|V(H_1\cap H_2)|=2$ and $V_0(H_1\cup H_2)\neq\emptyset$, this is at least $2|\overline{V(H_1\cup H_2)}|$, contradicting Proposition~\ref{lemma: there is no tight subgraph containing all three neighbours}. Hence, $H_1\cap H_2$ is composed of the two isolated vertices $v_1,v_2$, and $V_0(H_1)=V_0(H_2)=\emptyset$. So, $|E(H_1\cup H_2)|=2|V(H_1\cup H_2)|-2+\alpha_1+\alpha_2.$ Hence,
        \begin{equation}
        \label{equation for case 2a}
            \begin{split}
                |E(H)|&=|E(H_1\cup H_2)|+|E(H_3)|=(2|V(H_1\cup H_2)|-2+\alpha_1+\alpha_2)+(2|V(H_3)|-3+\alpha_3)\\
                &=2|V(H)|+2|V(H_1\cup H_2)\cap H_3|-5+\sum_{i=1}^3\alpha_i.
            \end{split}
        \end{equation}
        In particular, the intersection of $H_1\cup H_2$ and $H_3$ must indeed be the isolated vertex $v_3$. To see this, assume, by contradiction, that $|V(H_1\cup H_2)\cap H_3|\geq2$. Then $|E(H)|\geq2|V(H)|-1+\sum_{i=1}^3\alpha_i$.If $V_0(H)\neq\emptyset$, this is at least $2|\overline{V(H)}|+1$, contradicting the sparsity of $(G,\psi)$. If $V_0(H)=\emptyset$, then $\alpha_3\geq1$ (since $H_3+f_3$ is unbalanced), and so $|E(H)|\geq2|V(H)|=2|\overline{V(H)}|$, which contradicts Proposition~\ref{lemma: there is no tight subgraph containing all three neighbours}. So, $|V(H_1\cup H_2)\cap H_3|=1$ and
        \begin{equation}
        \label{edge set of union (2a)}
            |E(H)|=2|V(H)|-3+\sum_{i=1}^3\alpha_i.
        \end{equation}  
        Assume that $\alpha_1=\alpha_2=0$, so that $|E(H)|=2|V(H)|-3+\alpha_3$. Then, since all vertices of $H_1,H_2$ are free, $H_1,H_2$ are balanced blockers and, by Proposition~\ref{union of [...] and [...] is [...]}(iii), $H_1\cup H_2+f_1+f_2$ is near-balanced with base vertex $v_1$ (and with base vertex $v_2$). Since $H_1\cup H_2+f_1+f_2$ contains the 2-cycle $f_1,f_2$, it is near-balanced with gain $g$. So there is a gain $\psi'$ equivalent to $\psi$ such that $\psi'(e)\in\{\textrm{id},g,g^{-1}\}$ for all edges $e$ in $E(H_1\cup H_2)$ incident to $v_1$, and $\psi'(f)=\textrm{id}$ for all other edges $f\in E(H_1\cup H_2).$ In particular, $\left<H_1\cup H_2+f_1+f_2\right>=\left<g\right>$. Since $H_3+f_3$ contains the loop $f_3$ with gain $g$, it follows that $\left<H_1\cup H_2+f_1+f_2\right>\leq\left<H_3+f_3\right>$, and so $\left<H+f_1+f_2+f_3\right>\simeq\left<H_3+f_3\right>$. By Proposition~\ref{lemma: there is no tight subgraph containing all three neighbours} and Lemma~\ref{lemma:H+v=H+f_1+f_2}, $H_3+f_3$ must be proper near-balanced. Since it contains the loop $f_3$, it is near-balanced with base vertex $v_1$ and gain $g$. Recall that $H_1\cup H_2+f_1+f_2$ is also near-balanced with base vertex $v_1$ and gain $g$, so $H+f_1+f_2+f_3$ and $H+v$ are proper near-balanced with base vertex $v_1$ and gain $g$. But then $|E(H)|=2|V(H)|-3+\alpha_3=2|V(H)|-3+\alpha_k^j(H+f_1+f_2+f_3)$, which is a contradiction by Proposition~\ref{lemma: there is no tight subgraph containing all three neighbours}.
    
        Hence, $\alpha_1+\alpha_2\geq1$. In particular, $V_0(H)=\emptyset$, for otherwise, by Equation~\eqref{edge set of union (2a)}, $|E(H)|\geq2|\overline{V(H)}|$, which contradicts Proposition~\ref{lemma: there is no tight subgraph containing all three neighbours}. Since $H_3+f_3$ is not balanced, this implies that $\alpha_3\geq1$. Moreover, by Equation~\eqref{edge set of union (2a)} and Proposition~\ref{lemma: there is no tight subgraph containing all three neighbours}, $\sum_{i=1}^3\alpha_i\leq2$. So, $(\alpha_1,\alpha_2,\alpha_3)$ is one of $(0,1,1)$ and $(1,0,1)$. Without loss of generality, assume that $\alpha_1=0$, $\alpha_2=1$ and $\alpha_3=1$. By the definition of $\alpha_2,\alpha_3$, $j$ is odd and $\left<H_2+f_2\right>=\left<H_3+f_3\right>\simeq\mathbb{Z}_2$. Hence, $g=\gamma^{k/2}$ and each path from $v_1$ to $v_2$ in $H_2$ has gain $\textrm{id}$ or $g$. It follows that $\left<H_2\cup H_3+f_2+f_3\right>\simeq\mathbb{Z}_2$. However, 
        \begin{equation*}
        \begin{split}
            |E(H_2\cup H_3)|&=(2|V(H_2)|-2)+(2|V(H_3)|-2)\\
             &=2|V(H_2\cup H_3)|+2|V(H_2\cap H_3)|-4=2|V(H_2\cup H_3)|-2,
        \end{split}
        \end{equation*}
        contradicting Proposition~\ref{lemma: there is no tight subgraph containing all three neighbours} and Lemma~\ref{lemma:H+v=H+f_1+f_2}. Hence, the result holds.       
\end{proof}

\begin{proposition}
\label{pro: N(v)=2,v2 fixed}
    For $k\geq4$, let $\Gamma=\left<\gamma\right>\simeq\mathbb{Z}_k$ through the isomorphism defined by $\gamma\mapsto1$.
    For $2\leq j\leq k-2$, let $(G,\psi)$ be a $\mathbb{Z}_k^j$-gain tight $\Gamma$-gain graph with a free vertex $v$ of degree 3. Suppose that $v$ has no loop, and exactly two distinct neighbours $u,v_0$, of which only $u$ is free. Let $e_1,e_1'$ be the edges incident to $u$ and $v$, and let $e_2$ be the edge incident to $v_0$ and $v$. Suppose that either $j$ is even, or that the 2-cycle $e'_1e_1^{-1}$ does not have gain $\gamma^{k/2}$. Then there is an admissible 1-reduction at $v$.
\end{proposition}    

\begin{proof}
    Assume, without loss of generality, that $e_1,e_2,e_1'$ are directed from $v$ to $u$.  By Lemma~\ref{lemma: forests can have identity gains}, we may assume that $\psi(e_1)=\psi(e_2)=\textrm{id}$. Let $g=\psi(e_1')$. Let $(G_1,\psi_1),(G_2,\psi_2)$ be the graphs obtained from $G-v$ by adding, respectively, an edge $f_1=(u,v_0)$, and a loop $f_2$ at $u$ with gain $g$. Notice that, if there is already an edge $(v_1,v_2)\in E(G)$, $(G_1,\psi_1)$ is not a well-defined gain graph. 

    Assume that $H_2$ is a blocker for $(G_2,\psi_2)$ and, whenever $(u,v_0)\not\in E(G)$, assume that $H_1$ is a blocker for $(G_1,\psi_1)$. Since $H_2+f_2$ contains the loop $f_2$, $H_2$ is not a balanced blocker. Moreover, since $g\neq\gamma^{k/2}$ or $j$ is even, we cannot have $\left<H_2+f_2\right>\not\simeq\mathbb{Z}_2$ and $j$ is odd. So, if we show that $|V_0(H_2)|=0$, then $\alpha_k^j(H_2+f_2)\geq2$ by definition.

    Assume, by contradiction, that $v_0\in V(H_2)$. In particular, $H_2+f_2$ is not near-balanced, since $V_0(H_2)\neq\emptyset$. 
    Moreover, $\left<H_2+v\right>\simeq\left<H_2+f_2\right>$, since $v_0$ is fixed. Since $|V_0(H_2+v)|=|V_0(H_2+f_2)|$, it follows that $\alpha_k^j(H_2+v)=\alpha_k^j(H_2+f_2).$ But this contradicts Proposition~\ref{lemma: there is no tight subgraph containing all three neighbours}. Hence, $v_0\not\in V(H_2)$, and so $|V_0(H_2)|=\emptyset$. So, $\alpha_k^j(H_2+f_2)\geq2$ and $|E(H_2)|\geq2|V(H_2)|-1$. If $(u,v_0)\in E(G)$, then $|E(H_2+v_0)|=|E(H_2)|+1\geq2|V(H_2)|=2|\overline{V(H_2+v_0)}|$, which contradicts Proposition~\ref{lemma: there is no tight subgraph containing all three neighbours}. Hence, $(u,v_0)\not\in E(G)$, and $(G_1,\psi_1),H_1$ are well-defined. Let $H=H_1\cup H_2$ and $H'=H_1\cap H_2$. Notice that $H+f_1+f_2$ is neither balanced nor near-balanced, since it contains the loop $f_2$ and the fixed vertex $v_0$. Hence, by Corollary~\ref{cor:H+f_1+f_2 bal. or near-bal.}, $E(H')=\emptyset$. Then,
    \begin{equation*}
    \begin{split}
         |E(H)|&=(2|V(H_1)|-3+\alpha_k^j(H_1+f_1))+(2|V(H_2)|-3+\alpha_k^j(H_2+f_2))\\
         &=2|V(H)|+2|V(H')|-6+\alpha_k^j(H_1+f_1)+\alpha_k^j(H_2+f_2)\\
         &\geq2|V(H)|-4+\alpha_k^j(H_1+f_1)+\alpha_k^j(H_2+f_2)\\
         &\geq2|V(H)|-2=2|\overline{V(H)}|.
    \end{split}
    \end{equation*}
    
    This contradicts Proposition \ref{lemma: there is no tight subgraph containing all three neighbours}. Hence, there is an admissible 1-reduction at $v$. 
\end{proof}

\subsection{Applying a 1-reduction at a vertex with 3 distinct neighbours}
\begin{proposition}
    Let $\Gamma$ be a cyclic group of order $k\geq4$. For $2\leq j\leq k-2$, let $(G,\psi)$ be a $\mathbb{Z}_k^j$-gain tight $\Gamma$-gain graph with a free vertex $v$ of degree 3. Suppose that $v$ has no loop, and exactly three distinct neighbours $v_1,v_2,v_3$. Then there is an admissible 1-reduction at $v$.
\end{proposition}

\begin{proof}
    For $i=1,2,3$, let $e_i=(v,v_i)$ be the edges incident with $v$. We may assume, by Lemma~\ref{lemma: forests can have identity gains}, that $\psi(e_i)=\textrm{id}$ for $1\leq i\leq 3$. Let $f_1=(v_1,v_2),f_2=(v_2,v_3)$ and $f_3=(v_3,v_1)$. For $1\leq i\leq3$, let $(G_i,\psi_i)$ be obtained by applying a 1-reduction at $v$, during which we add the edge $f_i$ with gain $\textrm{id}$ and assume that $(G_i,\psi_i)$ has a blocker $H_i$. Let $H=H_1\cup H_2\cup H_3$, and $H'=H_1\cap H_2\cap H_3$. We will consider the following cases separately: $E(H_s\cap H_t)=\emptyset$ for at most two pairs of $s,t$; and $E(H_s\cap H_t)=\emptyset$ for all pairs $s,t$. In both cases, we show that there is a contradiction.

    \medskip
    \textit{\textbf{Case 1: }$E(H_s\cap H_t)=\emptyset$ for at most two pairs $s,t$.}\\
    Without loss of generality, we may assume $E(H_1\cap H_2)\neq\emptyset$. By Corollary~\ref{cor:H+f_1+f_2 bal. or near-bal.}, either $H_1\cup H_2+f_1+f_2$ is proper near-balanced or $v_2$ is fixed. If $H_1\cup H_2+f_1+f_2$ is near-balanced, say with base vertex $u$, then so is $H_1\cup H_2+v$, since every walk which contains $u$, from $v_1$ to $v_2$, from $v_2$ to $v_3$, and from $v_3$ to $v_1$ must have gain $\textrm{id},g$ or $g^{-1}$, for some $g\in\Gamma$. However, by Lemma~\ref{lemma: what i need for 1-red.}, $H_1\cup H_2$ is $(2,1)$-tight, which contradicts Proposition~\ref{lemma: there is no tight subgraph containing all three neighbours}. 
    
    Hence, we may assume that $v_2$ is fixed, and so $v_1,v_3$ are free. By the same argument as in the previous paragraph, it is easy to see that $E(H_1\cap H_3)=E(H_2\cap H_3)=\emptyset$. Hence, by Lemma~\ref{lemma: what i need for 1-red.},
    \begin{equation*}
    \begin{split}
        |E(H)|&=|E((H_1\cup H_2)\cup H_3)|=|E(H_1\cup H_2)|+|E(H_3)|\\
        &=(2|V(H_1\cup H_2)|-3+\alpha_k^j(H_1\cup H_2+f_1+f_2))+(2|V(H_3)|-3+\alpha_k^j(H_3+f_3))\\
        &=2|V(H)|+2|V((H_1\cup H_2)\cap H_3)|-6+\alpha_k^j(H_1\cup H_2+f_1+f_2)+\alpha_k^j(H_3+f_3)\\
        &\geq2|V(H)|-2+\alpha_k^j(H_1\cup H_2+f_1+f_2)+\alpha_k^j(H_3+f_3)=2|\overline{V(H)}|,
    \end{split}
    \end{equation*}
    since $v_1,v_3\in V((H_1\cup H_2)\cap H_3)$. This contradicts Proposition~\ref{lemma: there is no tight subgraph containing all three neighbours}.

    \medskip
    \textit{\textbf{Case 2: }$E(H_s\cap H_t)=\emptyset$ for all pairs $s,t$.}\\
    For simplicity, let $\alpha_i:=\alpha_k^j(H_i+f_i)$ for $1\leq i\leq3$. We have
    \begin{equation}
    \label{equation for N(v)=3}
        \begin{split}
|E(H)|&=\sum_{i=1}^3|E(H_i)|=2\sum_{i=1}^3|V(H_i)|-9+\sum_{i=1}^3\alpha_i\\
&=2[|V(H)|+\sum_{1\leq s\neq t\leq3}|V(H_s\cap H_t)|-|V(H')|]-9+\sum_{i=1}^3\alpha_i\geq2|V(H)|-3+\sum_{i=1}^3\alpha_i.
           \end{split}
    \end{equation}
    By the sparsity of $(G,\psi)$ and Proposition \ref{lemma: there is no tight subgraph containing all three neighbours}, $0\leq\sum_{i=1}^3\alpha_i\leq2$. Moreover, $|V(H_s\cap H_t)|\geq2$ for at most one pair $1\leq s\neq t\leq3$. Otherwise, $\sum_{1\leq s\neq t\leq3}|V(H_s\cap H_t)|-|V(H')|\geq5$, and so $|E(H)|\geq2|V(H)|+1$, contradicting the sparsity of $(G,\psi)$.
             
    First, let $\sum_{i=1}^3\alpha_i=0$ so that $|E(H)|\geq2|V(H)|-3$. Then, for each $1\leq i\leq 3$, $H_i+f_i$ is either balanced or it is $S_0(k,j)$ with $|V_0(H_i)|=1$. First, assume that each $H_i$ is a balanced blocker. If $|V(H_s\cap H_t)|=1$ for all pairs $1\leq s\neq t\leq3$, then $H+f_1+f_2+f_3$ is balanced: each path in $H_1$ (respectively $H_2$ and $H_3$) between $v_1$ and $v_2$ (respectively $v_2$ and $v_3$, and $v_1$ and $v_3$) has gain $\textrm{id}$. So, $H+v$ is also balanced. Since $|E(H)|\geq2|V(H)|-3$, this contradicts Proposition~\ref{lemma: there is no tight subgraph containing all three neighbours}. So, without loss of generality, assume that $|V(H_1\cap H_2)|=2$, and $|V(H_1\cap H_3)|=|V(H_2\cap H_3)|=1$, so that $|E(H)|\geq2|V(H)|-1$. If $V_0(H)\neq\emptyset$, then $|E(H)|\geq2|\overline{V(H)}|+1$, contradicting the sparsity of $(G,\psi)$. So $V_0(H)=\emptyset$. By Proposition~\ref{union of [...] and [...] is [...]}(iii), $H_1\cup H_2+f_1+f_2$ is near-balanced with base vertex $v_2$. Since each path in $H_3$ from $v_1$ to $v_3$ has gain $\textrm{id}$, it follows that $H+f_1+f_2+f_3$ is near-balanced with base vertex $v_2$. So $H+v$ is also near-balanced with base vertex $v_2$. Since $|E(H)|\geq2|V(H)|-1$, this contradicts Proposition~\ref{lemma: there is no tight subgraph containing all three neighbours}. 
              
    Now, assume that $H_i+f_i$ is $S_0(k,j)$ with $|V_0(H_i)|=1$ for some $1\leq i\leq3$. Without loss of generality, let $H_1+f_1$ be $S_0(k,j)$. If $|V(H_s\cap H_t)|\geq2$ for some pair $1\leq s\neq t\leq3$, then $|E(H)|\geq2|V(H)|-1=2|\overline{V(H)}|+1$, contradicting the sparsity of $(G,\psi)$. So $|V(H_s\cap H_t)|=1$ for all pairs $1\leq s\neq t\leq3$. In particular, $H_1+f_1,H_2+f_2,H_3+f_3$ cannot all be $S_0(k,j)$: otherwise, they all share a fixed vertex and, since $v_1,v_2,v_3\not\in V(H'),$ $|V(H_s\cap H_t)|\geq2$ for all $1\leq s\neq t\leq3$. So, without loss of generality, consider the following cases separately: $H_1+f_1,H_2+f_2$ are $S_0(k,j)$ and $H_3$ is a balanced blocker; $H_1+f_1$ is $S_0(k,j)$, and $H_2,H_3$ are balanced blockers. 

    First, assume that $H_1+f_1,H_2+f_2$ are $S_0(k,j)$ and $H_3$ is a balanced blocker. Let $n_1,n_2\in S_0(k,j)$ be such that $\left<H_1+f_1\right>\simeq\mathbb{Z}_{n_1},\left<H_2+f_2\right>\simeq\mathbb{Z}_{n_2}$. Since $|V(H_1\cap H_2)|=1$ and $H_1,H_2$ share the fixed vertex, $v_2$ is the fixed vertex. So,
    \begin{equation*}
    \begin{split}
    \left<H+f_1+f_2+f_3\right>&=\left<\psi(W):W\text{ is a closed walk in }H+f_1+f_2+f_3 \text{ not containing }v_2\right>\\
    &=\left<\psi(W):W\text{ is a closed walk in }H_1 \text{ or } H_2\text{ not containing } v_2 \text{, or in }H_3+f_3\right>\simeq\mathbb{Z}_l,
    \end{split}
    \end{equation*}
    where $l=\textrm{lcm}(n_1,n_2)\in S_0(k,j).$ So $H+f_1+f_2+f_3$ is $S_0(k,j)$, which contradicts the sparsity of $(G,\psi)$ and Proposition~\ref{lemma: there is no tight subgraph containing all three neighbours}, since $|E(H)|\geq2|\overline{V(H)}|-1$.

    Now, let $\left<H_1+f_1\right>\simeq\mathbb{Z}_n$ for some $n\in S_0(k,j)$, and $H_2,H_3$ be balanced blockers. Then the gain of $H+f_1+f_2+f_3$ is composed of the gain of every closed walk in $H_i$ not containing the fixed vertex, for $1\leq i\leq 3$, and the gain of every walk obtained by concatenating a walk from $v_1$ to $v_2$ (in $H_1$), a walk from $v_2$ to $v_3$ (in $H_2$), and a walk from $v_3$ to $v_1$ (in $H_3$). Since every walk from $v_1$ to $v_2$ has gain in $\mathbb{Z}_n$ (since $f_1$ has identity gain),  and every closed walk in $H_1$ has gain in $\mathbb{Z}_n$ (since $H_1\subset H_1+f_1$), and every closed walk in $H_2,H_3$, as well as every walk from $v_2$ to $v_3$ and from $v_3$ to $v_1$ has gain $\textrm{id}$, $\left<H+f_1+f_2+f_3\right>\simeq\mathbb{Z}_n$. By Lemma~\ref{lemma:H+v=H+f_1+f_2}, $H+v$ is $S_0(K,j)$. Since $|E(H)|\geq2|\overline{V(H)}|-1$, this is a contradiction, by the sparsity of $(G,\psi)$ and Proposition~\ref{lemma: there is no tight subgraph containing all three neighbours}. 

     So, let the triple $(\alpha_1,\alpha_2,\alpha_3)$ be one of $(1,0,0),(2,0,0),(1,1,0).$ In particular, since $\sum_{i=1}^3\alpha_i\geq1$, $|V(H_s\cap H_t)|=1$ for all $1\leq s\neq t\leq3$. Otherwise, $\sum_{1\leq s\neq t\leq3}|V(H_s\cap H_t)|-|V(H')|\geq4$, and so, by Equation~\eqref{equation for N(v)=3}, $|E(H)|\geq2|V(H)|$, contradicting Proposition~\ref{lemma: there is no tight subgraph containing all three neighbours}. Moreover, if $|V_0(H)|=1$, then $|E(H)|\geq|\overline{V(H)}|$ by Equation~\eqref{equation for N(v)=3}. This contradicts Proposition~\ref{lemma: there is no tight subgraph containing all three neighbours}, so $|V_0(H)|=0$.

    If the $(\alpha_1,\alpha_2,\alpha_3)=(1,1,0)$, then $j$ is odd, $\left<H_1+f_1\right>=\left<H_2+f_2\right>\simeq\mathbb{Z}_2$, and $H_3$ is a balanced blocker. Since $|V(H_s\cap H_t)|=1$ for all $1\leq s\neq t\leq3$, the gain of $H+f_1+f_3+f_2$ is given by the gain of each closed walk in $H_1+f_1,H_2+f_2$ and $H_3+f_3$, and the gain of every walk obtained by concatenating a walk from $v_1$ to $v_2$ (in $H_1$), a walk from $v_2$ to $v_3$ (in $H_2$), and a walk from $v_3$ to $v_1$ (in $H_3$). So, $\left<H+v\right>=\left<H+f_1+f_2+f_3\right>\simeq\mathbb{Z}_2$. Since $|E(H)|\geq2|V(H)|-2$, this contradicts Proposition~\ref{lemma: there is no tight subgraph containing all three neighbours} and the sparsity of $(G,\psi)$.
    
    So assume that $(\alpha_2,\alpha_3)=(0,0)$. Then $H_2\cup H_3+f_2+f_3$ is balanced, since $H_2\cap H_3$ is the isolated vertex $v_3$. Hence, $\left<H+v\right>=\left<H+f_1+f_2+f_3\right>=\left<H_1+f_1\right>$. Moreover, it's easy to see that  $H+f_1+f_2+f_3$ (and hence also $H+v$) is near-balanced whenever $H_1+f_1$ is near-balanced. Since $|V_0(G)|=0$, this implies that $\alpha_j^k(H+v)=\alpha_1$. Since $|E(H)|=2|V(H)|-3+\alpha_1$, this contradicts Proposition~\ref{lemma: there is no tight subgraph containing all three neighbours} and the sparsity of $(G,\psi)$. By contradiction, there is an admissible $1$-reduction at $v$.     
\end{proof}

\section{Final combinatorial results}\label{sec:final}

In this section, we prove the final combinatorial result of this paper (see Theorem~\ref{final theorem for rotation}), which characterises the infinitesimal rigidity of $\mathcal{C}_k$-generic frameworks. 
Throughout this Section, we let $5\leq k<1000$, or $k=4,6$. Recall that the study of the infinitesimal rigidity of a $\mathcal{C}_k$-generic framework can be split into the study of its $\rho_j$-symmetric isostaticity, for $0\leq j\leq k-1$. Recall also that $\rho_0$-,$\rho_1$- and $\rho_{k-1}$-symmetric isostaticity were already studied in \cite{SchulzeLP2024}. We state the result here, as Theorem~\ref{theorem: final theorem for k-fold symmetry pt.1}. The proof of Theorem~\ref{theorem: final theorem for k-fold symmetry pt.1} invokes induction on the order of the $\Gamma$-gain graph $(G,\psi)$, the base cases of which are given in Figure~\ref{Base graphs for 3-fold rotation}.

\begin{theorem}[Theorem 7.13 in \cite{SchulzeLP2024}]
\label{theorem: final theorem for k-fold symmetry pt.1}
    Let $\Gamma$ be a cyclic group of order $k\geq4$, and $(\tilde G,\tilde p)$ be a $\mathcal{C}_k$-generic framework. Let $(G,\psi)$ be the $\Gamma$-gain graph of $\tilde G$. Then the following hold.
    \begin{itemize}
        \item $(\Tilde{G},\tilde{p})$ is fully-symmetrically isostatic if and only if $(G,\psi)$ is $(2,0,3,1)$-gain-tight.
        \item $(\Tilde{G},\tilde{p})$ is $\rho_j$-symmetrically isostatic for $j=1,k-1$ if and only if $(G,\psi)$ is $(2,1,3,1)$-gain tight.
    \end{itemize} 
\end{theorem}

      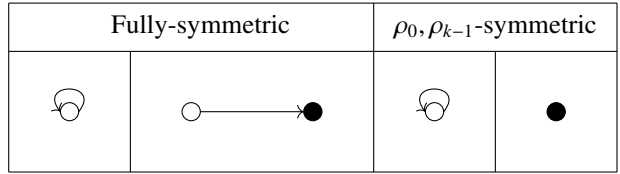
\begin{figure}[htp]
    \centering
    \begin{tikzpicture}
  [scale=.8,auto=left]

  \draw(-1,1.8) -- (9,1.8);
  \node at (2.15,1.4) {Fully-symmetric};
  \node at (7,1.4) {$\rho_0,\rho_{k-1}$-symmetric};
  \draw(-1,1) -- (9,1);
  \draw(-1,-1) -- (-1,1.8);
  
  \draw(0,0) circle (0.15cm);
  \draw[->] (0.15,0) .. controls (0.55,0.5) and (-0.55,0.5) .. (-0.15,0);

  \draw (1,-1) -- (1,1);

  \draw(2,0) circle (0.15cm);
  \draw[fill = black](4,0) circle (0.15cm);
  \draw[->] (2.15,0) -- (3.85,0);

  \draw (5,-1) -- (5,1.8);

  \draw(6,0) circle (0.15cm); 
  \draw[->] (6.15,0) .. controls (6.55,0.5) and (5.45,0.5) .. (5.85,0);

  \draw (7,-1) -- (7,1);

  \draw[fill = black](8,0) circle (0.15cm);

  \draw (9,-1) -- (9,1.8);
  \draw(-1,-1) -- (9,-1);
\end{tikzpicture}
    \caption{Base graphs for $k$-fold rotation for $\rho_0,\rho_1$ and $\rho_{k-1}$. All edges may be labelled freely, with the only restriction that loops must have non-identity gains.}
    \label{Base graphs for 3-fold rotation}
\end{figure} 

In a similar way, we use an inductive argument to prove the corresponding result for $2\leq j\leq k-2$. Namely, we will show that a $\mathcal{C}_k$-generic framework is $\rho_j$-symmetrically isostatic if and only if its underlying graph has a $\mathbb{Z}_k^j$-gain tight $\Gamma$-gain graph. Since our argument is inductive, we will be using the reduction moves described in Section~\ref{sec:red}, and so we first need to ensure that our $\mathbb{Z}_k^j$-gain graph has a vertex at which we may apply such moves.

\begin{lemma}[Lemma 7.1 in \cite{SchulzeLP2024}]
\label{lemma: there is a vertex of degree 2 or 3}
    Let $(G,\psi)$ be a $\Gamma$-gain graph with at least one free vertex. Let $s,t\in\mathbb{N}$ be the number of free vertices in $G$ of degree 2 and 3, respectively. Assume $(G,\psi)$ is $(2,0,0)$-tight. Then each free vertex of $G$ has degree at least 2. Moreover, if $G$ has a fixed vertex $v_0$, then $2s+t\geq\textrm{deg}(v_0)$.
\end{lemma}

Moreover, the case where $V_0(G)=\emptyset$ was already shown in \cite[Theorem 7.1]{Ikeshita} for odd $k\leq1000$, and in \cite{KCandST20} for $k=4,6$. Here, we unite the results, and state them as Theorem~\ref{Ikeshita theorem}. The proofs of Theorem~\ref{Ikeshita theorem} also apply an inductive argument. The base cases are a combination of disjoint unions of certain base graphs, which may be grouped into three classes. The first class is composed of the graphs in Figure~\ref{figure: base graphs for the proof by induction on k-fold symmetry for anti-symmetric motions.}. The second class consists of all $\mathbb{Z}_k^j$-gain tight 4-regular graphs which may be obtained from an $S(k,j)$ $\mathbb{Z}_k$-gain graph by adding an edge. The third class consists of all $\mathbb{Z}_k^j$-gain tight 4-regular graphs (with $j$ odd) which can be obtained from a $\mathbb{Z}_k$-gain graph $G$ with $\left<G\right>\simeq\mathbb{Z}_2$ by adding two edges (see Section 6.2 of \cite{Ikeshita} for details). When a fixed vertex is present, we will see that we obtain exactly one additional  connected component of a base graph, which is the isolated fixed vertex.

\begin{theorem}[\cite{KCandST20} and  \cite{Ikeshita}]
\label{Ikeshita theorem} 
Let $\Gamma$ be a cyclic group of order $k\geq4$. Assume that either $5\leq k\leq1000$ is odd or $k=4,6$, and let $(\tilde{G},\tilde p)$ be a $\mathcal{C}_k$-generic framework with underlying $\Gamma$-symmetric graph $\tilde G$. Let $(G,\psi)$ be the $\Gamma$-gain graph of $\tilde G$, and assume that $V_0(G)=\emptyset$. For $2\leq j\leq k-2$, $(\tilde G,\tilde p)$ is $\rho_j$-symmetrically isostatic if and only if $(G,\psi)$ is $\mathbb{Z}_k^j$-gain tight.
\end{theorem}

 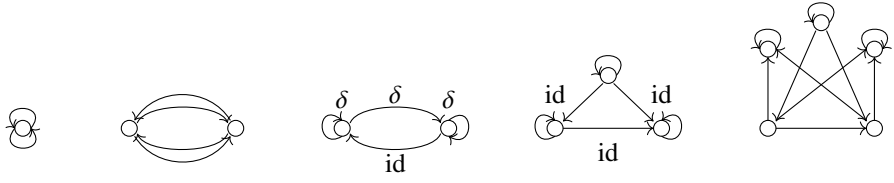
\begin{figure}[H]
    \centering
    \begin{tikzpicture}
  [scale=.7,auto=left]

  \draw(1,0) circle (0.15cm);  
  \draw[->] (1.15,0) .. controls (1.55,0.5) and (0.45,0.5) .. (0.85,0);
  \draw[->] (0.85,0) .. controls (0.45,-0.5) and (1.55,-0.5) .. (1.15,0);

  \draw(3,0) circle (0.15cm);
  \draw(5,0) circle (0.15cm);
  \draw[->] (3.15, 0.1) .. controls (3.3, 0.5) and (4.7, 0.5) .. (4.85, 0.1);
  \draw[->] (4.9, 0.15) .. controls (4.5, 0.8) and (3.5, 0.8) .. (3.1, 0.15);
  \draw[->] (4.85, -0.1) .. controls (4.7, -0.5) and (3.3, -0.5) .. (3.15, -0.1);
  \draw[->] (3.1, -0.15) .. controls (3.5, -0.8) and (4.5, -0.8) .. (4.9, -0.15);

  \draw(7,0) circle (0.15cm);
  \draw(9,0) circle (0.15cm);
  \node[above] at (7,0.2) {$\delta$};
  \node[above] at (9,0.2) {$\delta$};
  \node[above] at (8,0.3) {$\delta$};
  \node[below] at (8,-0.3) {$\textrm{id}$};
  \draw[->] (7.15, 0.1) .. controls (7.3, 0.5) and (8.7, 0.5) .. (8.85, 0.1);
  \draw[->] (8.85, -0.1) .. controls (8.7, -0.5) and (7.3, -0.5) .. (7.15, -0.1);
  \draw[->] (7,-0.15) .. controls (6.5,-0.45) and (6.5,0.55) .. (7,0.15);
  \draw[->] (9,0.15) .. controls (9.5,0.55) and (9.5,-0.45) .. (9,-0.15);

  \draw(11,0) circle (0.15cm);
  \draw(13,0) circle (0.15cm);
  \draw(12,1) circle (0.15cm);
  \node[below] at (12,-0.1) {$\textrm{id}$};
  \node[above] at (11,0.3) {$\textrm{id}$};
  \node[above] at (13,0.3) {$\textrm{id}$};
  \draw[->] (11.15,0) -- (12.85,0);
  \draw[->] (11.9,0.85) -- (11.15,0.13);
  \draw[->] (12.1,0.85) -- (12.85,0.13);
  \draw[->] (11,-0.15) .. controls (10.5,-0.45) and (10.5,0.55) .. (11,0.15);
  \draw[->] (13,-0.15) .. controls (13.5,-0.45) and (13.5,0.55) .. (13,0.15);
  \draw[->] (12.15,1) .. controls (12.55,1.5) and (11.45,1.5) .. (11.85,1);

  \draw(15,0) circle (0.15cm);
  \draw(17,0) circle (0.15cm);
  \draw(16,2) circle (0.15cm);
  \draw(15,1.5) circle (0.15cm);
  \draw(17,1.5) circle (0.15cm);
  \draw[->] (15.15,0) -- (16.85,0);
  \draw[->] (15.9,1.85) -- (15.15,0.13);
  \draw[->] (16.1,1.85) -- (16.85,0.13);
  \draw[->] (15,0.15) -- (15,1.35);
  \draw[->] (17,0.15) -- (17,1.35);
  \draw[->] (15.15,0.15) -- (16.85,1.5);
  \draw[->] (16.85,0.15) -- (15.15,1.5);
  \draw[->] (15.15,1.5) .. controls (15.55,2) and (14.45,2) .. (14.85,1.5);
  \draw[->] (16.85,1.5) .. controls (16.45,2) and (17.55,2) .. (17.15,1.5);
  \draw[->] (16.15,2) .. controls (16.55,2.5) and (15.45,2.5) .. (15.85,2);
\end{tikzpicture}
    \caption{Base graphs for $k$-fold rotation for $2\leq j\leq k-2$. All (unlabelled) edges of such graphs may be labelled freely, with the restrictions that loops must not have non-identity gains, the non-looped edges of the last graph are labelled $\text{id}$, and each graph must be $\mathbb{Z}_k^j$-gain tight.}
    \label{figure: base graphs for the proof by induction on k-fold symmetry for anti-symmetric motions.}
\end{figure}

The restriction $k\leq1000$ arises arises from the difficulty of computationally checking the rank of the corresponding orbit matrices for a growing list of base graphs. Both in \cite{Ikeshita} and in \cite{KCandST20}, it is  conjectured that this restriction may be dropped. For even $k\geq8$, there are  counterexamples to Theorem~\ref{Ikeshita theorem}, as we will see in Section~\ref{sec:even k}. Our final result relies on Theorem~\ref{Ikeshita theorem}. Hence, we must maintain all restrictions on $k$. For the cases where $k=4,6$, we need the following result.

\begin{lemma}
\label{lemma: there is an admissible reduction (rotation order 4,6)}
    For $k=4,6$, let $\Gamma=\left<\gamma\right>\simeq\mathbb{Z}_k$ through the isomorphism defined by letting $\gamma\mapsto1$. For $2\leq j\leq k-2$,  let
    $(G,\psi)$ be a $\mathbb{Z}_k^j$-gain tight $\Gamma$-gain graph with $V_0(G)=\{v_0\}$ and $|V(G)|\geq2$. Suppose that $\text{deg}(v_0)\geq1$. Then $(G,\psi)$ admits a reduction.
\end{lemma}
\begin{proof}
By Lemma \ref{lemma: there is a vertex of degree 2 or 3}, there is a free vertex in $V(G)$ of degree 2 or 3. We may assume that $G$ has no free vertex of degree 2 and no free vertex of degree 3 with a loop. Otherwise, we may apply a 0-reduction or loop-1-reduction to $(G,\psi)$. Further, we may assume that 
$k$ is even, $j$ is odd, and for all free vertices $v$ of degree 3, $v$ has exactly 2 distinct neighbours, one of which is $v_0$, and the 2-cycle $v$ forms with its free neighbour has gain $\gamma^{k/2}$ (see Figure~\ref{degree 2 vertex problem image pt. 2}). Otherwise, we may apply a 1-reduction to $(G,\psi)$, by Theorem \ref{theorem on 1-reductions}. Notice that, since $k$ is even and $j$ is odd, we must have $k=6$ and $j=3$. 

Let $v_1,\dots,v_t$ be the free vertices of degree 3 in $G$. For $1\leq i\leq t$ let $u_i$ be the free neighbour of $v_i$, and $e_i:=(u_i,v_0)$. By Lemma \ref{lemma: there is a vertex of degree 2 or 3}, $\textrm{deg}(v_0)\leq t$. So, if the edge $e_i$ is present for some $1\leq i\leq t,$ then $u_i$ must be a vertex of degree 3. Hence, we can apply a 2-vertex reduction at $u_i,v_i$. So, we may assume that $e_i\not\in E(G)$ for all $1\leq i\leq t$.

For $1\leq i\leq t,$ let $(G_i,\psi_i)$ be obtained from $(G,\psi)$ by removing $v_i$ and adding $e_i$ with gain $\textrm{id}$. We will show that, for some $1\leq i\leq t$, $(G_i,\psi_i)$ is an admissible 1-reduction. Assume, by contradiction, that for all $1\leq i\leq t$ there is a blocker $H_i$ for $(G_i,\psi_i)$. If there is some $1\leq i\leq t$ such that $\alpha_k^j(H_i+e_i)\geq1$, then
\begin{equation*}
    |E(H_i)|=2|V(H_i)|-3+\alpha_k^j(H_i+e_i)=2|\overline{V(H_i)}|-1+\alpha_k^j(H_i+e_i)\geq2|\overline{V(H_i)}|,
\end{equation*}
since $H_i$ contains the fixed vertex $v_0$. This contradicts Proposition~\ref{lemma: there is no tight subgraph containing all three neighbours}, so for all $1\leq i\leq t$, $\alpha_k^j(H_i+e_i)=0$ and 
$H_i$ is $(2,3)$-tight. By the definition of $\alpha_k^j$, it follows that, for all $1\leq i\leq t$, $H_i+e_i$ is either balanced or $S_0(6,3)$. Since $S_0(6,3)=\{3\}$, it follows that $\left<H_i+e_i\right>$ is either $\left\{\text{id}\right\}$ or $\{\text{id},\gamma^2,\gamma^4\}$.

Moreover, for each $1\leq i\neq s\leq t$, $v_s\not\in V(H_i)$. To see this, suppose, by contradiction, that $v_s\in V(H_i)$. 
Since $\left<H_i+e_i\right>$ is either $\left\{\text{id}\right\}$ or $\{\text{id},\gamma^2,\gamma^4\}$, it cannot contain the 2-cycle $(v_s,u_s)(u_s,v_s)$ of gain $\gamma^3$. Hence, there is an edge $e$ incident to $v_s,u_s$ such that $e\not\in E(H_i)$. It is easy to see that, since $H_i$ i $(2,3)$-tight, all of its vertices have degree 2 in $H_i$ (see, for instance, the proof of Lemma 7.1(i) in \cite{SchulzeLP2024}). In particular, $u_s$ has degree 2 in $H_i$, so two edges incident to $u_s$ lie in $H_i$. Then, $|E(H_i+e)|=|E(H_i)|+1=2|V(H_i)|-2=2|\overline{V(H_i+e)}|$, since $v_0\in V(H_i)$. This contradicts Proposition~\ref{lemma: there is no tight subgraph containing all three neighbours}, so $v_s\not\in V(H_i)$ for all $1\leq i\neq s\leq t$.

\medskip
\textit{\textbf{Claim: }$E(H_i\cap H_s)=\emptyset$ and $V(H_i\cap H_s)=\{v_0\}$ for all $1\leq i\neq s\leq t$}.

\medskip

\textit{Proof. }Choose some $1\leq i\neq s\leq t$. First, assume by contradiction that $E(H_i\cap H_s)\neq\emptyset$. By the proof of Lemmas~\ref{lemma:bal.int.} and~\ref{lemma:alpha>1}, we can see that $|E(H_i\cup H_s)|=2|V(H_i\cup H_s)|-3=2|\overline{V(H_i\cup H_s)}|-1$. But then, 
\begin{equation*}
    |E(H_i\cup H_s+v_i+v_s)|=|E(H_i\cup H_s)|+6=2|\overline{V(H_i\cup H_s)}|+5=2|\overline{V(H_i\cup H_s+v_i+v_s)}|+1,
\end{equation*}
contradicting the sparsity of $(G,\psi)$. So $E(H_i\cap H_s)=\emptyset$ for all $1\leq i\neq s\leq t$. 

Now, if $V(H_i\cap H_s)\neq\{v_0\}$, then $H_i\cap H_s$ contains a free vertex, and so $|E(H_i\cup H_s)|=|E(H_i)|+|E(H_s)|=2|\overline{V(H_i\cup H_s)}|+2|\overline{V(H_i\cap H_s)}|-2\geq2|\overline{V(H_i\cup H_s)}|.$ This contradicts Proposition~\ref{lemma: there is no tight subgraph containing all three neighbours}, so $V(H_i\cap H_s)=\{v_0\}$. Since $i,s$ were arbitrary, the claim holds. $\square$

\medskip
Let $H:=\bigcup_{i=1}^tH_i$. By the Claim,
\begin{equation*}
        |E(H)|=\sum_{i=1}^t|E(H_i)|
        =2\sum_{i=1}^t|V(H_i)|-3t
        =2(|V(H)|+(t-1))-3t
        =2|V(H)|-t-2.
\end{equation*}
So, $H':=H+v_1+\dots+v_t$ satisfies $|E(H')|=2|\overline{V(H'}|.$ This implies that there is no edge $e\in E(G)\setminus E(H')$ that joins two vertices in $V(H')$ and $H'$ is $(2,0,0)$-tight.

Next, we show that $H'$ is a connected component of $G$. Clearly, $H'$ is connected. Suppose $G$ has a non-empty subgraph $G'$ such that $V(G)$ is the disjoint union of $V(H')$ and $V(G')$. Let $d(H',G')$ be the number of edges joining a vertex in $H'$ with one in $G'$. We aim to show that $d(H',G')=0$. Let $\alpha\geq0$ be such that $|E(G')|=2|V(G')|-\alpha=2|\overline{V(G')}|-\alpha$. Then,
\begin{equation*}
\begin{split}
    |E(G)|&=|E(H')|+|E(G')|+d(H',G')=2|\overline{V(H')}|+2|V(G')|-\alpha+d(H',G')\\
    &=2|\overline{V(G)}|-\alpha+d(H',G')=|E(G)|-\alpha+d(H',G'),
\end{split}
\end{equation*}
so $\alpha=d(H',G').$ Since every vertex in $G'$ has degree at least $4$ in $G$, $4|V(G')|\leq \sum_{v\in V(G')} deg_G(v)=2|E(G')|+d(H',G')=4|V(G')|-2\alpha+\alpha=4|V(G')|-\alpha,$ and so $d(H',G')=\alpha=0,$ as required. 

\medskip
Finally, consider $H_1$ and let $n,m$ be the vertices of degree 2 and 3 in $H_1$, respectively. Let $\hat{\rho},\rho_{\text{min}}$ be the average degree and minimum attainable degree of $H_1$, respectively. Since $H_1$ is $(2,3)$-tight, $|V(H_1)|\hat{\rho}=4|V(H_1)|-6$. Moreover, $\rho_{\text{min}}$ is attained when all vertices of $H_1$ have degree 2,3 or 4, and hence $|V(H_1)|\rho_{\text{min}}=4|V(H_1)|-2n-m.$ Since $\rho_{\text{min}}\leq\hat{\rho}$, we have $2n+m\geq6.$ Hence, there are at least three vertices of degree 2 or 3 in $H_1$. If two of the vertices are $v_0,v_1$, there is still a free vertex $w$ in $H_1$ of degree 2 or 3. Since $H_1$ is a connected component of $G$, it follows that $w$ has degree 2 or 3 in $G$. But this contradicts our assumption that the only free vertices of degree 2 or 3 in $G$ are $v_1,\dots,v_t$. Hence, our result holds by contradiction.
\end{proof}

We now prove the main result of this paper.

\begin{theorem}
\label{theorem: final theorem for k-fold symmetry pt.2}
    For $k\geq4$, let $\Gamma=\left<\gamma\right>\simeq\mathbb{Z}_k$ through the isomorphism defined by letting $\gamma\mapsto1$.
    Assume that either $5\leq k\leq1000$ is odd or $k=4,6$, and let $(\tilde{G},\tilde p)$ be a $\mathcal{C}_k$-generic framework with underlying $\Gamma$-symmetric graph $\tilde G$. Let $(G,\psi)$ be the $\Gamma$-gain graph of $\tilde G$. For $2\leq j\leq k-2$, $(\tilde G,\tilde p)$ is $\rho_j$-symmetrically isostatic if and only if $(G,\psi)$ is $\mathbb{Z}_k^j$-gain tight.
\end{theorem}

\begin{proof}
    We use induction on $|V(G)|$. If $V(G)=V_0(G)=\{v_0\}$, then $(G,\psi)$ is an isolated fixed vertex, and so it is easy to see that $(\tilde G,\tilde p)$ is $\rho_j$-symmetrically isostatic. The $\Gamma$-liftings of the graphs in Figure~\ref{figure: base graphs for the proof by induction on k-fold symmetry for anti-symmetric motions.} were shown to have $\rho_j$-symmetrically isostatic realisations in \cite{Ikeshita}. The base cases of our induction argument are exactly the disjoint combinations of the base graphs given in \cite{Ikeshita} (see the paragraph after Lemma~\ref{lemma: there is a vertex of degree 2 or 3}), and of the isolated fixed vertex.
    
    We may assume that $\overline{V(G)}\neq\emptyset$ (since otherwise we obtain a base graph). Assume further that the statement is true for all graphs on at most $t$ vertices, for some integer $t\geq1$, that $|V(G)|=t+1$, and that $G$ is not a base graph. 
    
    If $V_0(G)=\emptyset$, or if $V(G)$ has an isolated fixed vertex, then the graph $(G',\psi')$ obtained from $(G,\psi)$ by removing  its fixed vertex (if it has one), is $\mathbb{Z}_k^j$-gain tight. By Theorem~\ref{Ikeshita theorem}, $(\tilde{G'},\tilde{p}|_{V(G')})$ is $\rho_j$-symmetrically isostatic. Since $O_j(G,\psi,p)=O_j(G',\psi',p|_{V(G')})$, $(\tilde{G},\tilde{p})$ is also $\rho_j$-symmetrically isostatic. So, we may assume that $G$ has a connected component $H$ which contains a fixed vertex, and which is not a base graph. Hence, the fixed vertex has degree at least $1$. 

    If $|\overline{V(G)}|=1$, then $V(G)=\{v_0,v\}$, where $v_0$ is a fixed vertex and $v$ is free, and $E(G)$ is composed of a loop $e$ at $v$, and an edge between $v$ and $v_0$. Since $(G,\psi)$ is $\mathbb{Z}_k^j$-gain tight, if $k=6$ and $j=3$, then $e$ does not have gain $\gamma^{k/2}$. Moreover, $G$ is not $S_0(k,j)$. We may apply a loop-1-reduction at $v$ to obtain a $\mathbb{Z}_k^j$-gain tight graph $(G',\psi')$ on $t$ vertices. By the inductive hypothesis, every $\mathcal{C}_k$-generic realisation of $\tilde{G'}$ is $\rho_j$-symmetrically isostatic. Let $(\tilde{G}',\tilde{q}')$ be a $\mathcal{C}_k$-generic realisation of $\tilde{G}'$. By Lemma~\ref{lemma:rank ext.}, there is a $\mathcal{C}_k$-symmetric realisation $(\tilde G,\tilde q)$ of $\tilde G$ which is $\rho_j$-symmetrically isostatic. Then, since $(\tilde G,\tilde p)$ is $\mathcal{C}_k$-generic, it is also $\rho_j$-symmetrically isostatic.
    
    So, we may assume that $|\overline{V(G)}|\geq2$. If $k=4,6$, by Lemma~\ref{lemma: there is an admissible reduction (rotation order 4,6)}, there is a $\mathbb{Z}_k^j$-gain tight graph $(G',\psi')$ on at most $t$ vertices obtained from $(G,\psi)$ by applying a reduction (exactly $t$ if we apply a 0-reduction, loop-1-reduction or 1-reduction, and exactly $t-1$ if we apply a 2-vertex reduction). By induction, every $\mathcal{C}_k$-generic realisation of $\tilde{G'}$ is $\rho_j$-symmetrically isostatic. Moreover, if we apply a loop-1-reduction at a vertex $v$ which removes a loop $e$, by the sparsity of $(G,\psi)$, the following hold: if $k=6,j=3$, then $e$ does not have gain $\gamma^{k/2}$; if the vertex incident to $v$ is fixed, call it $v_0$, then the graph spanned by $v,v_0$ is not $S_0(k,j)$. So conditions (C2) and (C3) in Lemma~\ref{lemma:rank ext.} hold.
    
    Let $\tilde{q}'$ be a $\mathcal{C}_k$-generic configuration of $\tilde{G'}$, which also satisfies the condition (C1) in Lemma~\ref{lemma:rank ext.} if the move applied is a 1-reduction. Notice that such a configuration does exist, since small symmetry-preserving perturbations of the points of a $\mathcal{C}_k$-generic framework maintain $\mathcal{C}_k$-genericity. By Lemma~\ref{lemma:rank ext.} there is a $\mathcal{C}_k$-symmetric realisation $(\tilde G,\tilde q)$ of $\tilde G$ which is $\rho_j$-symmetrically isostatic. By $\mathcal{C}_k$-genericity, $(\tilde G,\tilde p)$ is also $\rho_j$-symmetrically isostatic.
    
    So, assume that $k$ is odd. By Lemma~\ref{lemma: there is a vertex of degree 2 or 3}, $H$ has a free vertex $v$ of degree $2$ or $3$. If $v$ has degree $2$, or if it has degree $3$ with a loop, then we may apply a 0-reduction or loop-1-reduction at $v$ to obtain a $\mathbb{Z}_k^j$-gain tight graph $(G',\psi')$ on $t$ vertices. Moreover, if $v$ has a loop, and the vertex incident to $v$ is fixed, call it $v_0$, then the graph spanned by $v,v_0$ is not $S_0(k,j)$. By the inductive hypothesis, all $\mathcal{C}_k$-generic realisations of $\tilde{G'}$ are $\rho_j$-symmetrically isostatic. Then, our result holds by Lemma~\ref{lemma:rank ext.}. So, assume that $v$ has degree $3$ and no loop. Then, by Theorem~\ref{theorem on 1-reductions}, there is a $\mathbb{Z}_k^j$-tight graph $(G',\psi')$ on $t$ vertices obtained by applying a $1$-reduction at $v$. By the inductive hypothesis, all $\mathcal{C}_k$-generic realisations of $\tilde{G'}$ are $\rho_j$-symmetrically isostatic. Let $\tilde{q}'$ be a $\mathcal{C}_k$-generic realisation of $\tilde{G'}$ which satisfies condition (C1) of Lemma~\ref{lemma:rank ext.}. Then, our result holds by Lemma~\ref{lemma:rank ext.}. 
\end{proof}

We finally have our main combinatorial characterisation for $\mathcal{C}_k$, which is a direct result of Proposition \ref{necessary conditions for higher rotation} and Theorems \ref{theorem: final theorem for k-fold symmetry pt.1} and \ref{theorem: final theorem for k-fold symmetry pt.2}. 

\begin{theorem}
\label{final theorem for rotation}
Let $\Gamma$ be a cyclic group of order $k\geq4$. Assume that either $5\leq k\leq1000$ is odd or $k=4,6$, and let $(\tilde G,\tilde p)$ be a $\mathcal{C}_k$-generic framework with underlying $\Gamma$-symmetric graph $\tilde G$. Let $(G,\psi)$ be the $\Gamma$-gain graph of $\tilde G$. Then, $(\tilde G,\tilde p)$ is infinitesimally rigid if and only if:
\begin{itemize}
\item $(G,\psi)$ has a $(2,0,3,1)$-gain tight spanning subgraph; and
\item $(G,\psi)$ has a $(2,1,3,1)$-gain tight spanning subgraph; and
\item $(G,\psi)$ has a $\mathbb{Z}_k^j$-gain tight spanning subgraph for $2\leq j\leq k-2$.
\end{itemize}
\end{theorem}

\section{Rotation groups of even order at least 8}
\label{sec:even k}
In this section, we provide, for all even $|\Gamma|\geq 8$, examples of $\Gamma$-gain graphs that satisfy all conditions of Theorem~\ref{final theorem for rotation}, but whose $\mathcal{C}_{|\Gamma|}$-generic lifting frameworks are still not infinitesimally rigid. 

Let $k:=|\Gamma|\geq4$ be even, and let $G$ be the multigraph with exactly one free vertex $v$, which is free, and two loops $f_1,f_2$ at $v$ (see  Figure~\ref{fig: counterexample for C_8}(a)). Let $\gamma$ be the generator of $\Gamma$ which corresponds to 1 in $\mathbb{Z}_k$. Let $\psi:E(G)\rightarrow\Gamma$ be defined by letting $\psi(f_1)=\gamma$ and $\psi(f_2)=\gamma^3$. If $k\geq6$, $(G,\psi)$ is a well-defined $\Gamma$-gain graph. Moreover, if $k\geq8$, then $(G,\psi)$ is $\mathbb{Z}_k^j$-gain tight for all $2\leq j\leq k-2$. Since $G-f_1$ is both $(2,0,3,1)$-gain tight and $(2,1,3,1)$-gain tight, $(G,\psi)$ satisfies all three conditions of Theorem~\ref{final theorem for rotation}. Let $\tilde G$ be the $\Gamma$-lifting of $(G,\psi)$. We will show that no $\mathcal{C}_k$-symmetric realisation of $\tilde G$ is infinitesimally rigid. Further, we show that all $\mathcal{C}_k$-symmetric realisations of $\tilde G$ have a $\rho_{k/2}$-symmetric infinitesimal motion.

Take an arbitrary $\mathcal{C}_k$-symmetric realisation $(\tilde G,\tilde p)$ of $\tilde G$. By definition, the realisation of the vertices in $V(G)$ form a regular $k$-gon. Moreover, it is easy to see that the vertices of the $k$-gon alternate between vertices of the two partite sets of a bipartite graph (see e.g. Figure~\ref{fig: counterexample for C_8}(b) for the case when $k=8$), as no odd cycles are created. Clearly, the framework is $\mathcal{C}_k$-generic. It is also well known that such a framework has an `in-out' infinitesimal motion $m$ which, for $\tau(\delta)=C_k$, satisfies the system of equations
\begin{equation*}
    m(\delta^tv)=\begin{cases}
        C_k^tm(v) & \text{if } t \text{ is even}\\
        -C_k^tm(v) & \text{if } t \text{ is odd},
    \end{cases}
\end{equation*}
where $v$ is an arbitrary vertex of $\tilde{G}$ (here, $m(v)$ is a vector on the line from the origin to $p_v$), and $0\leq t\leq k-1$ (see e.g. \cite{Whitely84}).  Equivalently, for all $v\in V(\tilde{G})$ and $0\leq t\leq k-1$,
\begin{equation*}
    m(\gamma^tv)=\cos(\pi t)C_k^tm(v)=\cos(-\pi t)C_k^tm(v)=\exp(-\pi it)C_k^tm(v)=\overline{\rho_{k/2}(\gamma^t)}C_k^tm(v).
\end{equation*}
So, $m$ is a $\rho_{k/2}$-symmetric infinitesimal motion.

This example may be extended to the case in which the $\Gamma$-gain graph has a fixed vertex. Let $G$ be a multigraph with exactly two free vertices $u,v$, and one fixed vertex $v_0$. Let the edge set of $G$ be composed of two loops $f_1,f_2$ at $u$, one loop $f_3$ at $v$, and the edges $e_1=(u,v)$ and $e_2=(v,v_0)$ (see Figure~\ref{fig: counterexample for C_8}(c)). Let $\psi:E(G)\rightarrow\Gamma$ be defined by letting $\psi(f_1)=\gamma,\psi(f_2)=\gamma^3,\psi(f_3)=\gamma^2$, and $\psi(e_1)=\psi(e_2)=\text{id}$. Similarly as in the previous examples, $(G,\psi)$ is well-defined for all $k\geq6$. Moreover, it has the following spanning subgraphs: $G-f_1-f_3$, which is $(2,0,3,1)$-gain tight; $G-f_1$, which is $(2,1,3,1)$-gain tight; and $G-f_3$, which is $\mathbb{Z}_k^j$-gain tight for all $2\leq j\leq k-2$, provided $k\geq8$. Hence, for $k\geq8$, $(G,\psi)$ satisfies all conditions in Theorem~\ref{final theorem for rotation}. However, its $\Gamma$-covering $\tilde G$ has no infinitesimally rigid $\mathcal{C}_k$-symmetric realisation.

To see this, take a $\mathcal{C}_k$-generic realisation of $\tilde G$, and call it $(\tilde G,\tilde p)$. Since this is an extension of the previous example, $(\tilde G,\tilde p)$ still contains a regular $k$-gon $P$, and the graph induced by the vertices of $P$ is bipartite. In addition, $(G,\psi)$ contains two regular $k/2$-gons, $P_1$ and $P_2$, such that all vertices of $P_1,P_2$ are adjacent to the origin, and they are adjacent with the vertices of $P$ as shown in Figure~\ref{fig: counterexample for C_8}(d). Then, the infinitesimal motion from the previous example extends to an infinitesimal motion $m$ of $(\tilde G,\tilde p)$ which rotates $P_1$ and $P_2$ clockwise and anti-clockwise, respectively.
Similarly as in the previous example, it is easy to see that $m$ is a $\rho_{k/2}$-symmetric infinitesimal motion of $(\tilde G,\tilde p)$. (It is easy to check that the rank of the $\rho_{k/2}$-orbit matrix is at most 3, and so $\ker O_j(G,\psi,p)\neq\emptyset$. For details, see \cite{PhDThesis}.)

\begin{figure}[H]
    \centering
        \begin{tikzpicture}[scale=0.75]
        \draw(0,0) circle (0.15cm);
    \draw[->] (0,0.15) .. controls (0.45,0.65) and (0.45,-0.65) .. (0,-0.15);
    \draw[->] (0,-0.15) .. controls (-0.45,-0.65) and (-0.45,0.65) .. (0,0.15);
    \node[below] at (-0.7, 0.3){$\gamma$};
  \node[below] at (0.7, 0.5){$\gamma^3$};
  \node[below] at (0,-2.85){(a)};
    \end{tikzpicture}
    \begin{subfigure}[b]{.3\textwidth}
   \centering
      \begin{tikzpicture} 
        \newdimen\R
   \R=1.5cm
\draw (0:\R) \foreach \x in {45,90,...,360} {  -- (\x:\R) };
\foreach \x in {45,135,...,360} \draw[->,gray] (\x:\R) -- (\x: 1.9cm);
\foreach \x in {90,180,...,360} \draw[->,gray] (\x:\R) -- (\x: 1.1cm);
\draw (0:\R) -- (135:\R);
\draw (45:\R) -- (180:\R);
\draw (90:\R) -- (225:\R);
\draw (135:\R) -- (270:\R);
\draw (180:\R) -- (315:\R);
\draw (225:\R) -- (360:\R);
\draw (270:\R) -- (45:\R);
\draw (315:\R) -- (90:\R);
   \foreach \x/\l/\p in
     { 45/{(2,3,1)}/above,
      90/{(2,1,3)}/above,
      135/{(1,2,3)}/left,
      180/{(1,3,2)}/below,
      225/{(3,1,2)}/below,
      270/{(3,1,2)}/below,
      315/{(3,1,2)}/below,
      360/{(3,2,1)}/right
     }
     \node[inner sep=1pt,circle,draw,fill] at (\x:\R) {};
     \node[below] at (0,-2.3) {(b)};
      \end{tikzpicture}
\end{subfigure}
\begin{tikzpicture}[scale = 0.75]
        \draw(0,1) circle (0.15cm);
        \draw[fill=black](2,1) circle (0.15cm);
        \draw(-0.5,1.8) circle (0.15cm);
  \draw[gray, ->] (-0.35,1.8) .. controls (0.05,2.3) and (-1.05,2.3) .. (-0.65,1.8);
  \draw[->] (0,1.15) .. controls (0.45,1.65) and (0.45,0.35) .. (0,0.85);
  \draw[->] (0,0.85) .. controls (-0.45,0.35) and (-0.45,1.65) .. (0,1.15);
  \draw[gray, ->] (0, 1.15) -- (-0.38, 1.65);
  \draw[gray, ->] (-0.33, 1.8) -- (1.85, 1.1);
  \node[below] at (-0.7, 1.3){$\gamma$};
  \node[below] at (0.65, 1.5){$\gamma^3$};
  \node[above] at (-0.5, 2.1){$\gamma^2$};
  \node[below] at (1,-1.85) {(c)};
    \end{tikzpicture}
    \begin{subfigure}[b]{.3\textwidth}
   \centering
      \begin{tikzpicture}
        \newdimen\R
   \R=1.5cm
\draw (0:\R) \foreach \x in {45,90,...,360} {  -- (\x:\R) };
\foreach \x in {45,90,...,360} \draw[gray] (\x:\R) -- (\x-20:2.3cm) --(0:0cm);
\foreach \x in {45,135,...,360} \draw[->,gray] (\x:\R) -- (\x: 1.75cm);
\foreach \x in {90,180,...,360} \draw[->,gray] (\x:\R) -- (\x: 1.25cm);
\foreach \x in {90,180,...,360} \draw[->, gray] (\x-20:2.3cm) -- (\x-13: 2.3cm);
\foreach \x in {45,135,...,360} \draw[->, gray] (\x-20:2.3cm) -- (\x-27: 2.3cm);
\draw[gray] (70:2.3cm) -- (160:2.3cm) -- (250:2.3cm) -- (340:2.3cm)-- (70:2.3cm);
\draw[gray] (25:2.3cm) -- (115:2.3cm) -- (205:2.3cm) -- (295:2.3cm) -- (25:2.3cm);
\draw (0:\R) -- (135:\R);
\draw (45:\R) -- (180:\R);
\draw (90:\R) -- (225:\R);
\draw (135:\R) -- (270:\R);
\draw (180:\R) -- (315:\R);
\draw (225:\R) -- (360:\R);
\draw (270:\R) -- (45:\R);
\draw (315:\R) -- (90:\R);
   \foreach \x/\l/\p in
     { 45/{(2,3,1)}/above,
      90/{(2,1,3)}/above,
      135/{(1,2,3)}/left,
      180/{(1,3,2)}/below,
      225/{(3,1,2)}/below,
      270/{(3,1,2)}/below,
      315/{(3,1,2)}/below,
      360/{(3,2,1)}/right
     }
     \node[inner sep=1pt,circle,draw,fill] at (\x:\R) {};
     \node[inner sep=1pt,circle,draw,fill] at (0:0cm) {};
      \foreach \y in {25,70,...,385} \node[inner sep=1pt,circle,draw,fill] at (\y:2.3cm) {};
      \node at (0,-2.5) {(d)};
      \end{tikzpicture}
\end{subfigure}
    \caption{(a,c) show $\Gamma$-gain graphs with $\mathcal{C}_8$-symmetric frameworks (b,d), respectively. Though (a,c) satisfy the conditions in Theorem~\ref{final theorem for rotation}, (b,d) are $\rho_4$-symmetrically flexible. Here, $\gamma$ denotes the generator of $\Gamma$ which corresponds to rotation by $\pi/4$.}
    \label{fig: counterexample for C_8}
\end{figure}
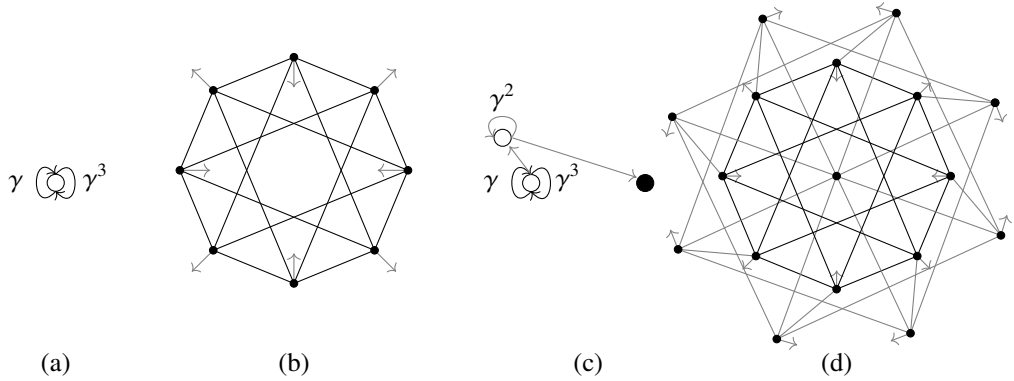

It would be interesting to search for further counterexamples that are not based on bipartite graphs. We also invite the reader to extend the combinatorial characterisations we have established for odd order cyclic groups in this paper to cyclic groups of odd order greater than 1000. The key issue here is to check the infinitesimal rigidity of the relevant base graphs. For further open questions on the infinitesimal rigidity of incidentally symmetric frameworks, see \cite{PhDThesis,SchulzeLP2024}.

\end{document}